\documentclass[12pt]{amsart}

\usepackage{amsfonts,amssymb,enumerate,bm,hhline,makecell,array,mathtools}
\usepackage{mathrsfs}
\usepackage{comment}
\usepackage[normalem]{ulem}
\usepackage[colorlinks=true,citecolor=blue!70!black, urlcolor=blue!70!black, pagebackref, linkcolor=blue!70!black]{hyperref}
\usepackage{tikz-cd}
\usepackage{amscd}
\usepackage[shortlabels]{enumitem}
\setlist[itemize]{noitemsep,nolistsep}
\usepackage{tensor}
\usepackage{tikz}
\usetikzlibrary{matrix,arrows}
\usepackage{extarrows}
\usepackage[toc,page]{appendix}
\usepackage[all,ps,cmtip,rotate]{xy} 
\usepackage{color}

\makeatletter
\newcommand*{\da@rightarrow}{\mathchar"0\hexnumber@\symAMSa 4B }
\newcommand*{\da@leftarrow}{\mathchar"0\hexnumber@\symAMSa 4C }
\newcommand*{\xdashrightarrow}[2][]{%
  \mathrel{%
    \mathpalette{\da@xarrow{#1}{#2}{}\da@rightarrow{\,}{}}{}%
  }%
}
\newcommand{\xdashleftarrow}[2][]{%
  \mathrel{%
    \mathpalette{\da@xarrow{#1}{#2}\da@leftarrow{}{}{\,}}{}%
  }%
}
\newcommand*{\da@xarrow}[7]{%
  \sbox0{$\ifx#7\scriptstyle\scriptscriptstyle\else\scriptstyle\fi#5#1#6\m@th$}%
  \sbox2{$\ifx#7\scriptstyle\scriptscriptstyle\else\scriptstyle\fi#5#2#6\m@th$}%
  \sbox4{$#7\dabar@\m@th$}%
  \dimen@=\wd0 %
  \ifdim\wd2 >\dimen@
    \dimen@=\wd2 %   
  \fi
  \count@=2 %
  \def\da@bars{\dabar@\dabar@}%
  \@whiledim\count@\wd4<\dimen@\do{%
    \advance\count@\@ne
    \expandafter\def\expandafter\da@bars\expandafter{%
      \da@bars
      \dabar@ 
    }%
  }%  
  \mathrel{#3}%
  \mathrel{%   
    \mathop{\da@bars}\limits
    \ifx\\#1\\%
    \else
      _{\copy0}%
    \fi
    \ifx\\#2\\%
    \else
      ^{\copy2}%
    \fi
  }%   
  \mathrel{#4}%
}
\makeatother

\newcommand{\nc}{\newcommand}
\nc{\sA}{{\mathsf{A}}}
\nc{\sB}{{\mathsf{B}}}
\nc{\sC}{{\mathsf{C}}}
\nc{\sD}{{\mathsf{D}}}
\nc{\sE}{{\mathsf{E}}}
\nc{\sF}{{\mathsf{F}}}
\nc{\sG}{{\mathsf{G}}}
\nc{\sH}{{\mathsf{H}}}
\nc{\sI}{{\mathsf{I}}}
\nc{\sJ}{{\mathsf{J}}}
\nc{\sK}{{\mathsf{K}}}
\nc{\sL}{{\mathsf{L}}}
\nc{\sM}{{\mathsf{M}}}
\nc{\sN}{{\mathsf{N}}}
\nc{\sO}{{\mathsf{O}}}
\nc{\sP}{{\mathsf{P}}}
\nc{\sQ}{{\mathsf{Q}}}
\nc{\sR}{{\mathsf{R}}}
\nc{\sS}{{\mathsf{S}}}
\nc{\sT}{{\mathsf{T}}}
\nc{\sU}{{\mathsf{U}}}
\nc{\sV}{{\mathsf{V}}}
\nc{\sW}{{\mathsf{W}}}
\nc{\sX}{{\mathsf{X}}}
\nc{\sY}{{\mathsf{Y}}}
\nc{\sZ}{{\mathsf{Z}}}

\def\cD{\mathscr{D}}

\def\cO{\mathscr{O}}

\def\cP{\mathscr{P}}
\def\cH{\mathscr{H}}

\def\cS{\mathscr{S}}
\def\cC{\mathscr{C}}
\def\cM{\mathscr{M}}

\def\cU{\mathscr{U}}
\def\cV{\mathscr{V}}

\def\cX{\mathscr{X}}
\def\cY{\mathscr{Y}}
\def\cZ{\mathscr{Z}}

\def\bv{\ensuremath{\mathbf v}}

\def\bp{\mathbf p}

\def\Z{{\bf Z}}

\def\C{{\bf C}}

\def\R{{\bf R}}

\def\P{{\bf P}}
\def\PP{{\bf P}}

\def\phi{\varphi}

\DeclareMathOperator{\Bl}{Bl}

\DeclareMathOperator{\Cl}{Cl}

\DeclareMathOperator{\disc}{disc}
\def\div{\mathop{\rm div}\nolimits}

\DeclareMathOperator{\GL}{GL}
\DeclareMathOperator{\Gr}{Gr}

\DeclareMathOperator{\Hom}{Hom}

\def\Im{\mathop{\rm Im}\nolimits}

\DeclareMathOperator{\KKK}{K3}

\DeclareMathOperator{\LG}{LG}

\DeclareMathOperator{\lin}{\underset{\mathrm lin}{\equiv}}

\DeclareMathOperator{\Pic}{Pic}

\DeclareMathOperator{\rk}{rk}

\DeclareMathOperator{\Spec}{Spec}

\DeclareMathOperator{\rank}{rank}

\DeclareMathOperator{\Sym}{Sym}

\def\eps{\varepsilon}

\def\vide{\varnothing}

\def\setminus{\smallsetminus}

\def\isom{\simeq}

\def\lra{\longrightarrow}
\def\llra{\hbox to 10mm{\rightarrowfill}}
\def\lllra{\hbox to 15mm{\rightarrowfill}}

\def\llla{\hbox to 10mm{\leftarrowfill}}
\def\lllla{\hbox to 15mm{\leftarrowfill}}

\def\dra{\dashrightarrow}

\def\lhra{\ensuremath{\lhook\joinrel\relbar\joinrel\to}}

\DeclareMathOperator{\isomto}{\stackrel{{}_{\scriptstyle\sim}}{\to}}

\def\hk{hyper-K\"ahler}

\newtheorem{lemm}{Lemma}[section]
\newtheorem{theo}[lemm]{Theorem}
\newtheorem{coro}[lemm]{Corollary}
\newtheorem{prop}[lemm]{Proposition}

\theoremstyle{remark}
\newtheorem{defi}[lemm]{Definition}
\newtheorem{rema}[lemm]{Remark}

\makeatletter
\@namedef{subjclassname@2020}{%
  \textup{2020} Mathematics Subject Classification}
\makeatother

\makeatletter
\def\@tocline#1#2#3#4#5#6#7{\relax
  \ifnum #1>\c@tocdepth % then omit
  \else
    \par \addpenalty\@secpenalty\addvspace{#2}%
    \begingroup \hyphenpenalty\@M
    \@ifempty{#4}{%
      \@tempdima\csname r@tocindent\number#1\endcsname\relax
    }{%
      \@tempdima#4\relax
    }%
    \parindent\z@ \leftskip#3\relax \advance\leftskip\@tempdima\relax
    \rightskip\@pnumwidth plus4em \parfillskip-\@pnumwidth
    #5\leavevmode\hskip-\@tempdima
      \ifcase #1
       \or\or \hskip 1em \or \hskip 2em \else \hskip 3em \fi%
      #6\nobreak\relax
    \dotfill\hbox to\@pnumwidth{\@tocpagenum{#7}}\par
    \nobreak
    \endgroup
  \fi}
\makeatother

\def\bw#1#2{\textstyle{\bigwedge\hskip-0.9mm^{#1}}\hskip0.2mm{#2}}

\def\id{\mathsf{id}}

\title{On the geometry of singular EPW cubes}

\begin{document}

\author[F.~Rizzo]{Francesca Rizzo}
\address{Universit\'e   Paris Cit\'e and Sorbonne Universit\'e, CNRS,   IMJ-PRG, F-75013 Paris, France}
 \email{{\tt francesca.rizzo@imj-prg.fr}}

 %\keywords{\red{...}}
\thanks{This project has received funding from the European
Research Council (ERC) under the European
Union's Horizon 2020 research and innovation
programme (Project HyperK --- grant agreement 854361).
}

\date{\today}

\begin{abstract}
 {EPW cubes form a locally complete family of smooth projective \hk\  varieties of dimension 6, constructed by Iliev--Kapustka--Kapustka--Ranestad.\ Their construction and behavior share a lot of similarities with the double EPW sextics constructed by O'Grady.\ 
Adapting the methods of O'Grady, we construct a projective smooth small resolution of singular EPW cubes.}
\end{abstract}

\maketitle

  \tableofcontents
% \setcounter{tocdepth}{1}

%%%%%%%%%%%%%%%%%%%%%
\section{Introduction}

Very few geometric constructions of locally complete families of smooth projective \hk\ varieties are known.\ One of these families is that of so-called  EPW cubes, constructed by Iliev--Kapustka--Kapustka--Ranestad in~\cite{ikkr}.\ Their construction goes as follows.

Let $V_6$ be   a $6$-dimensional complex vector space $V_6$.\ We endow $\bw3V_6$  with the conformal symplectic structure given by wedge product and we
 denote by $\LG(\bw3V_6)$ the  Lagrangian Grassmannian that parametrizes Lagrangian subspaces   $A\subset \bw3V_6$.\  For any $[A]\in \LG(\bw3V_6)$ and nonnegative integer $k$, the authors of \cite{ikkr} define  the degeneracy loci
 \begin{equation*}\label{eqn:EPwQuartStrat}
\sZ^{\ge k}_A \coloneqq \{[U_3]\in \Gr(3, V_6) \mid \dim(A\cap (\bw2U_3 \wedge V_6))\ge k\}, \quad \sZ^k_A \coloneqq \sZ^{\ge k}_A \setminus \sZ^{\ge k+1}_A
\end{equation*}
in the Grassmannian $\Gr(3, V_6)  $.\ 

When $A$ has no decomposable vectors (this means that the intersection $\P(A)\cap \Gr(3,V_6)$, inside $\P(\bw3V_6)$, is empty, and happens exactly when $[A]$ is outside  an irreducible divisor $\Sigma\subset \LG(\bw3V_6)$), the scheme $\sZ^{\ge 2}_A$ is an integral normal sixfold whose singular locus is the threefold~$\sZ^{\ge 3}_A$.\ Moreover, the locus $\sZ^{\ge 4}_A$ is finite, and it is empty for $[A]$ outside  another irreducible divisor $\Gamma\subset \LG(\bw3V_6)$. 

In \cite[Theorem~5.7]{dk2}, when $[A]\notin\Sigma$ (so that $A$ has no decomposable vectors), the authors constructed a double cover
\begin{equation}\label{eqn:EPWcube}
    g\colon \widetilde \sZ^{\ge 2}_A \lra \sZ^{\ge 2}_A
\end{equation}
 branched over $\sZ^{\ge 3}_A$.\ 
When moreover $[A]\notin\Gamma$ (namely when $\sZ^{\ge 4}_A = \varnothing$), the authors of   \cite{ikkr} prove that {\em the scheme $\widetilde \sZ^{\ge 2}_A$ is a smooth \hk\ variety of $\KKK^{[3]}$-type with a polarization of square $4$ and divisibility $2$,} called an EPW cube.\ This is an important result because, as explained above, there are very few explicit constructions of this type.\ It justifies a more extensive study of this construction.

{The aim of this paper is to study the singular sixfold $\widetilde \sZ^{\ge 2}_A$ when $[A]\in \Gamma\smallsetminus\Sigma$.\ Adapting the methods of \cite{og} (where the author worked with double EPW sextics), we show the following:
\begin{enumerate}
    \item[(\rm{a})] The scheme $\sZ^{\ge 4}_A=\{z_1,\dots,z_r\}$ is  finite and smooth, and equals $\sZ^{4}_A$; moreover, if $[A]$ is general in $\Gamma$, then $\sZ^{4}_A$ is a singleton ($r=1$).
    \item[(\rm{b})] The singular set of the sixfold $\widetilde \sZ^{\ge 2}_A$ is equal to $g^{-1}(\sZ^{4}_A)=\{g^{-1}(z_1),\dots,g^{-1}(z_r)\}$.\ For each $i\in \{1,\dots,r\}$, the tangent cone to $\widetilde \sZ^{\ge 2}_A$ at the point $g^{-1}(z_i)$ is isomorphic to the (affine) cone over the incidence variety $I\subset \P^3\times (\P^3)^\vee$.\ Therefore, the exceptional divisor  of the blowup  $\widetilde X_A\to \widetilde \sZ^{\ge 2}_A$ along $g^{-1}(\sZ^{4}_A)$ is the disjoint union of smooth divisors $E_1,\dots,E_r$, all isomorphic to $I$ and with normal bundles $\cO_{E_i}(-1,-1)$.\ For any choice~$\eps = (\eps_1, \dots, \eps_r)$ of contractions  of each~$E_i$ onto either $\P^3$ or $(\P^3)^\vee$, we obtain an (analytic) small resolution $\sX_A^\eps\to \widetilde \sZ^{\ge 2}_A$ with exceptional locus a disjoint union of $r$ copies of $\P^3$.
   % and \emph{locally} we have two (analytic) small resolutions of each singular point in $\widetilde \sZ^{\ge 2}_A$ with fiber $\Proj^3$.
%\item For a sufficiently small open (in the classical topology) $\cV\subset \LG(\bw3V_6)$ containing $[A]$, and for any choice of $\epsilon$, the  family of EPW cubes parameterized by $\cV$ has a simultaneous resolution of singularities  with fiber $\sX_A^\epsilon$ over $A$.
\item[(\rm{c})] There exists a choice of $\eps$ such that $\sX^\eps_A$ is 
 a projective smooth  quasi-polarized \hk\ sixfold   with a projective contraction $\sX^\eps_A\to\widetilde \sZ^{\ge 2}_A$ of $r$ copies of $\P^3$.\end{enumerate}
}

Concerning Item~(a), the finiteness and smoothness of $\sZ^{\ge 4}_A$ 
%(already stated in \cite{dk2}) 
are proved in Theorem~\ref{theorem:ikkr} by generalizing the argument of \cite[Lemma~2.8]{ikkr}.\ 
  In Section~\ref{sec:ParamCount}, following \cite[Section~2.1]{og}, we show that~$\sZ^{\ge 4}_A$ consists of a single point for $[A]$ general in  $\Gamma$.
{Regarding Item~(b), we show in Proposition~\ref{prop:LocalStructureZ2} and Lemma~\ref{lem:GeneralHyper}  that, analytically, the singularity of $\sZ^{\ge 2}_A$ at a point of $\sZ^4_A$ is the cone over a general hyperplane section, in the space of ($4\times 4$)-matrices,  of the subscheme $\cS_{\le 2}$ of symmetric matrices of rank~$\le 2$.\ By the uniqueness of the double covers of   degeneracy loci constructed in   \cite[Theorem~3.1]{dk2}, the singularity of~$\widetilde \sZ^{\ge 2}_A$ at a point of $g^{-1}(\sZ^4_A)$ is the singularity of the ``universal double cover'' of this general hyperplane section of~$\cS_{\le 2}$, studied in details in Section~\ref{sec:CanonicalDoubleCover} (in other words, it is the cone over the incidence variety $I$ mentioned above).\

The construction of the small analytic  resolutions   $\sX_A^\eps\to\widetilde \sZ^{\ge 2}_A$ is analogous to the one in~\cite[Sections~3.1 and 3.2]{og}.\ We do it in Section~\ref{se4}.}

The most subtle part of our results is Item~(c), that is,  proving that there exists a small \emph{projective}  resolution of $\widetilde \sZ^{\ge 2}_A$ (as in the double EPW sextic case, if $\sZ^4_A$ has more than one point,   we do not expect   the small resolutions $\sX_A^\eps\to \widetilde \sZ^{\ge 2}_A$ to be all projective -- that is, K\"ahler).\ For   double EPW sextics, this was done by O'Grady  by explicitly constructing  a projective $\KKK$ surface $S_A$ and (under a mild generality hypothesis) an isomorphism between the Hilbert square $S_A^{[2]}$ and one of these analytic resolutions.\ We cannot hope for an analogous construction  in our case, because there are no associated K3 surfaces, and $\sY_A$ cannot be realized, even birationally, as a moduli space of sheaves over a K3 surface (we prove this in Proposition~\ref{prop:NoAssociatedK3}).%, so that we do not have a geometric candidate for a small {projective}  resolution.

However, using the properties of the period map for EPW cubes and the surjectivity of the period map for \hk\ sixfolds,  we produce in Proposition~\ref{prop:ProjSmallRes} a smooth projective \hk\ sixfold $\sY_A$ with a big and nef line bundle $H$ that induces a small contraction $\sY_A\to \widetilde\sZ^{\ge 2}_A$; the pair $(\sY_A,H)$ is a deformation of standard (smooth) EPW cubes.\ Finally, using results from~\cite{kp23}, we show in Theorem~\ref{theo:projContraction} that $\sY_A$ is actually a divisorial contraction of the blowup~$\widetilde \sX_A$, hence one of the  small analytic resolutions $\sX_A^\eps $.

%\rem{Explain what's in the appendix.}
 {In Appendix~\ref{appendix}, we describe the Heegner divisors in the period domain of \hk\ varieties of $\KKK^{[3]}$-type with polarization of square 4 and divisibility 2, and show that the   divisors %$\Delta$, 
 $\Gamma$ and $\Sigma$ of $\LG(\bw3V_6)$ map onto   distinct Heegner divisors.}

\subsection*{Acknowledgement} I would like to express my deep gratitude to my advisor, Olivier Debarre, for suggesting   this problem to me and for his guidance.\  
I thank Alexander Kuznetsov and Claire Voisin for their help with Theorem~\ref{theo:projContraction}.\  
I also thank Pietro Beri, Francesco Denisi, Franco Giovenzana, Alexander Kuznetsov, and Rita Pardini for useful discussions. 
%\rem{Thank Sasha and Claire}

\section{The degeneracy loci $\sZ_A^{\ge k}$}
The following description of the degeneracy loci $\sZ_A^{\ge k} \subset \Gr(3,V_6)$ associated with a Lagrangian subspace $A\subset \bw3V_6$ with no decomposable vectors, namely vectors of the form $v_1\wedge v_2\wedge v_3$, is stated in \cite[Theorem~5.6]{dk2}.

\begin{theo}\label{theorem:ikkr}
If a Lagrangian subspace $A\subset \bw3V_6$ has no decomposable vectors, the following properties hold. 
\begin{itemize}
\item[\textnormal{(a)}] $\sZ_A^{\ge 1} $ is an integral normal quartic hypersurface in $\Gr(3,V_6)$.
\item[\textnormal{(b)}] $\sZ_A^{\ge 2}$ is the singular locus of $\sZ_A^{\ge 1} $; 
it is an integral normal Cohen--Macaulay sixfold of degree~$480$.
\item[\textnormal{(c)}] $\sZ_A^{\ge 3}$ is the singular locus of $\sZ_A^{\ge 2} $; 
it is an integral normal Cohen--Macaulay threefold of degree~$4944$.
\item[\textnormal{(d)}] $\sZ_A^{\ge 4}$ is the singular locus of  $\sZ_A^{\ge 3}$; 
{it} is finite and smooth, and is empty for $A$ general.
\item[\textnormal{(e)}] $\sZ_A^{\ge 5}$ is empty.
\end{itemize}
\end{theo}

 However, in the reference given in \cite{dk2} for this result  (\cite[Proposition~2.6 and Corollary~2.10]{ikkr}),   the smoothness of $\sZ_A^{k}$ is only studied for $k\in\{1,2,3\}$, so (d) and (e) need to be proven.\ This is the aim of this section.

 We fix some notation.\ Let $[U_0]$ be a point of $\Gr(3, V_6)$.\ We set
  \begin{equation}\label{deftu}
     T_{U_0}\coloneqq \bw2 U_0\wedge V_6.
 \end{equation}
  This is a (10-dimensional) Lagrangian subspace of $\bw3 V_6$ and the projective tangent space to~$\Gr(3, V_6)$ at  $[U_0]$ is $\P(T_{U_0})\subset \P(\bw3 V_6)$.\ If   we choose a subspace $U_\infty\subset V_6$ complementary to~$U_0$, then $T_{U_0}$ decomposes as $\bw3{U_0}\oplus\Hom (U_0, U_\infty)$ and the Zariski tangent space to $\Gr(3, V_6)$ at  $[U_0]$ is isomorphic to $\Hom (U_0, U_\infty)$.
  
 We also set
 $$\cC_{U_0}\coloneqq \PP(T_{U_0})\cap \Gr(3, V_6),$$
 the intersection of $\PP(T_{U_0})$ with the space of decomposable vectors of $\LG(\bw3V_6)$.
 This is a 5-dimensional cone with vertex $ [\bw3{U_0}]$ over the smooth Segre fourfold $$\cM_1\simeq\P(U_0^\vee)\times \P(U_\infty)\subset \P(\Hom(U_0, U_\infty))$$  of morphisms  of rank $1$.

 \begin{proof}[Proof (of Theorem~\ref{theorem:ikkr}(e))]

 This is very easy to prove.\ Assume $[U_0]$ is in  $\sZ_A^{\ge 5}$, so that $A\cap T_{U_0}$ has dimension at least $5$.\ 
Then, $\P(A\cap T_{U_0})$ has dimension at least~$4$, hence must meet the fivefold $\cC_{U_0}$ in  $\P(T_{U_0})=\P^9$, which contradicts the fact that $A$ has no decomposable vectors.\ Therefore,  $\sZ_A^{\ge 5}$ is empty.
 
 It remains to prove Theorem~\ref{theorem:ikkr}(d).\ This will be done in Section~\ref{localde}.\ Before that, we prove   general preliminary results that will also be used later on in this paper.\ 
 \end{proof}
 \subsection{Tangent space to $\Gr(3, V_6)$}\label{subsec:TgGr36}
 In this section, we study the map  $T\colon \Gr(3, V_6)\to \LG(\bw3V_6)$ that takes a point $[U]$ to the Lagrangian subspace $[T_U]$.

%We fix a point $U_0\in \sZ^{k}_{A}\cap \Gr(3, V_5)$ and choose $U_\infty$ such that $A\cap T_{U_\infty} = 0$ and $V_6 = U_0\oplus U_\infty$.\ Moreover, we consider bases $(v_1, v_2, v_3)$ of $U_0$ and $(v_4, v_5, v_6)$ of $U_\infty$, and write elements $B$ of $\Hom{U_0}{U_\infty}$ as matrices $B=(b_{ij})_{1\le i,j\le 3}$ in these bases.\\

%An affine neighborhood of $U_0$ in $\Gr(3, V_6)$ is given by $$\Uu = \{[U]\mid U\cap U_\infty = 0\}\simeq \Hom{U_0}{U_\infty}.$$ We will write $T_0$ instead of $T_{U_0}$ and $T_\infty$ instead of $T_{U_\infty}$.

\begin{prop}\label{prop:tgtGr36}
The tangent map to $T$ at a point $[U_0] $ gives   an identification 
$$
\Theta\colon T_{\Gr(3, V_6), [U_0]} \stackrel{\sim}{\lra} H^0(\P(T_{U_0}), I_{\cC_{U_0}}(2))
$$
of $9$-dimensional vector spaces, where $I_{\cC_{U_0}}$ is the ideal sheaf of the subscheme $\cC_{U_0}\subset \P(T_{U_0})$.
\end{prop}

%Consider the morphism 
%$$
 %\Psi\colon \cV \times \cU \lra \Sym^2 \!T_{U_0}^\vee 
%$$
%sending $([A], [U])$ to the quadratic form $q_{U}-q_{A}$, where $q_{U}$ and $q_{A}$ are the quadratic forms on~$T_{U_0}$ associated with the Lagrangian subspaces $T_U$ and $A $ (which are elements of $\cV$)  respectively.\ Then the intersection $\cZ^{\ge k} \cap (\cV\times \cU )$ is equal to $ \Psi^{-1}(\Sigma_k)$, where $\Sigma_k\subset \Sym^2\! T_{U_0}^\vee$ is the corank~$k$ degeneracy locus.

\begin{proof}
 Consider the affine open subsets  
$$\cU \coloneqq \{[U ]\in \Gr(3, V_6)\mid U \cap U_\infty =0\}\subset\Gr(3, V_6), $$
neighborhood of $[U_0]$, and 
$$\cV \coloneqq \{[A]\in \LG(\bw3V_6)\mid A\cap T_{U_\infty} =0\}\subset\LG(\bw3V_6). $$
 We get identifications $\Hom(U_0, U_\infty)\simeq \cU$ and $\Sym^2\!T_{U_0}^\vee\simeq \cV$ by sending a morphism to its graph and a quadratic form on $ T_{U_0}$, viewed as a symmetric morphism $T_{U_0}\to T_{U_0}^\vee$, to the graph of the composition of this morphism with the inverse of the  isomorphism $T_{U_\infty}\isomto T_{U_0}^\vee$ given by the symplectic form. 

The morphism $T\vert_\cU\colon \cU\to\cV$  then
 corresponds to a morphism  
 \begin{align*}
\Hom(U_0,U_\infty)  &\lra \Sym^2(T_{U_0}^\vee) 
 \\
         [U] &\longmapsto [q_U],
 \end{align*}
 where $q_U$ is described in \cite[Lemma~2.7]{ikkr} as a polynomial in the entries  of the matrix $ (b_{ij})_{1\le i, j \le 3}$ representing $U$ in fixed bases of $U_0$ and $U_\infty$.\ In particular, if we write an element of $T_{U_0} =\bw3 U_0 \oplus \Hom(U_0, U_\infty)$ as $(m, M)$, the linear part in the $b_{ij}$ of $q_U(m, M)$ is given by $\sum_{ij} b_{i,j} M^{i,j}$,
 where $M^{i,j}$ is the determinant of the $2\times 2$ submatrix of $M$ in which we remove the column $i$ and the row $j$. 

 Therefore, the image of the tangent space $T_{\cU, [U_0]}$ by $T_{\Theta, [U_0]}$ is generated by the cofactors~$M^{i,j}$, namely, this image is, in $\Sym^2(T_{U_0}^\vee) = H^0(\P(T_{U_0}),  \cO (2))$, the linear subsystem of quadrics in~$\P(T_{U_0})$ that vanish on the set of matrices of rank 1.  \end{proof}

%\subsection{Tangent cones to $\sZ^{k}_{A}$}\label{subsec:TgGr35andZk}
\subsection{The restriction map $r_K$}\label{subsec:TgGr35andZk}
Let $A\subset \bw3V_6$ be a Lagrangian subspace and let $[U_0]$ be a point of $\sZ^{k}_{A}$, so that $K \coloneqq A\cap T_{U_0}$ has dimension $k$.\ The restriction map
\begin{equation}\label{defrk}
r_K\colon H^0(\P(T_{U_0}), I_{\cC_{U_0}}(2))\lra H^0(\P(K), \cO(2))\isom \Sym^2\!K^\vee
\end{equation}
will play a very important role for us.\ The following proposition is crucial for our argument.

\begin{prop}\label{prop23}
    Let  $A\subset \bw3V_6$ be a Lagrangian subspace with  no decomposable vectors and let $[U_0]$ be a point of $\sZ^{k}_{A}$, with $k\in\{1,2,3,4\}$.\ With the notation above,
    \begin{itemize}
        \item[\rm{(a)}] when $k\in\{1,2,3\}$, the map $r_K$ is surjective with kernel of dimension $9-k(k+1)/2$,
        \item[\rm{(b)}] when $k=4$, the map $r_K$ is injective with image a general hyperplane in  $H^0(\P(K), \cO(2))$.
    \end{itemize}
\end{prop}

The meaning of ``general'' in Item~(b) is the following: a hyperplane in $H^0(\P(K), \cO(2))\isom \Sym^2\!K^\vee$ is defined by a nonzero element of $\Sym^2\!K$ and we want the corresponding tensor to be of maximal rank $k=4$. 

The proof will show that the proposition is true more generally for any linear subspace $K\subset T_{U_0}$ of dimension $k\le 4$ with no decomposable vectors.

\begin{proof}
Recall that  $\cC_{U_0}$ is the cone with vertex $\bv\coloneqq [\bw3{U_0}]$ over the smooth fourfold $\cM_1\subset \P(\Hom(U_0, U_\infty))$ of morphisms  of rank $1$.\ Since the latter is not contained in any hyperplane, so is $\cC_{U_0}$, hence its projective Zariski tangent space   at $\bv$ is the whole space $\P^9$.\ Any quadric  containing $\cC_{U_0}$  is therefore a cone with vertex $\bv$.

Since $A$ contains no decomposable vectors, $\P(K)$ is disjoint from $\cC_{U_0}$.\ In particular, $\bv\notin \P(K)$, so we may, upon projecting from $\bv$, consider $\P(K)$ as a linear subspace  of $\P(\Hom(U_0, U_\infty))$ (as above, $U_\infty\subset V_6$ is a subspace complementary to~$U_0$).\ All in all, this proves that
  $r_K$ identifies with the restriction morphism
\begin{equation*}\label{defrpk}
  r'_K\colon H^0(\PP(\Hom (U_0, U_\infty)), I_{\cM_1}(2)) \longrightarrow H^0(\P(K), \cO(2)),  
\end{equation*}
where $\P(K)\cap \cM_1=\vide$.\

That being said, Item~(a) is actually \cite[Lemma~2.8]{ikkr}, so we may assume $k=4$ and we need to prove that $r'_K$ is injective with general image.\ 
Suppose by contradiction that there exists   a  nonzero  element  
 in the kernel of $r'_K$.\ Its zero-locus    $Q\subset \PP(\Hom (U_0, U_\infty))$ is a quadric   that contains the disjoint $\P(K)=\PP^3$ and  $\cM_1$.\ This contradicts the following lemma and proves the injectivity of $r'_K$.
 
 \begin{lemm}
    Let $\cM_1\subset \P(\Hom(U_0, U_\infty))=\P^8$ be the smooth Segre fourfold of morphisms  of rank $1$ and let $Q_0\subset \P^8$ be a quadric containing $\cM_1$.\ Any $\P^3\subset Q_0$ meets $\cM_1$.
\end{lemm}

\begin{proof}
 The space of quadrics in $\P^8$ containing $\cM_1$ is the $\P^8$ spanned by the 9 maximal minors of elements of $\Hom(U_0, U_\infty)$, and the group $\GL(U_0) \times \GL(U_\infty)$
acts on it with only three orbits:
\begin{enumerate}
    \item[(a)]  quadrics of (maximal) rank 9, represented for example by $M^{1,1}+M^{2,2}+M^{3,3}$;
    \item[(b)] quadrics of rank 6, represented  for example by $M^{2,2}+M^{3,3}$;
\item[(c)] quadrics of rank 4, represented for example by $M^{3,3}$,
\end{enumerate}
where $M^{i,j}$ is the $(i,j)$-minor of an element $M=(x_{ij})_{1\le i,j\le3}$ of $\Hom(U_0, U_\infty)$ (in each case, it is enough to find one quadric containing $\cM_1$ of the correct rank).

Since the action of $\GL(U_0) \times \GL(U_\infty)$ fixes $\cM_1$, we can suppose that the quadric $Q_0$ is of one in the previous list.

In case (a), the Lefschetz hyperplane theorem shows that $H^{2i}(Q_0, \Z)\simeq H^{2i}(\PP^8, \Z) = \Z h^i$, where $h\subset \PP^8$ is a hyperplane section, for each $i\le 3$.\ Therefore,  $[\cM_1]=mh^3\vert_{Q_0}\in H^{6}(Q_0, \Z)$ for some  $m\in \Z_{>0}$.\ Thus, for each $\PP^3\subset Q_0$, we have 
$$[\PP^3]\cdot [\cM_1]=m (h\vert_{\P^3})^3=m\ne 0.
$$
Hence, each $\PP^3\subset Q_0$ meets $\cM_1$.\

  In case (b), the quadric 
$$Q_0 =  M^{2,2}+M^{3,3}=x_{11}(x_{22}+x_{33})-x_{12}x_{21}-x_{13}x_{31}$$ has rank 6, hence is a  cone over a smooth quadric 
$$Q_1 \subset \P^5=\P\begin{pmatrix}
   x_{11}& x_{12}&x_{13}\\x_{21}&x_{22}&0\\ x_{31}&0&0
\end{pmatrix} $$ with
vertex $$\P^2=(x_{11}=x_{22}+x_{33}=x_{12}=x_{21}=x_{13}=x_{31}=0) .$$
One sees that $\cM_1\cap \P^2$ is the smooth conic $x_{22}^2+x_{23}x_{32}=0$ and that $\cM_1\cap Q_1$ contains the planes 
$$P_1=\P(x_{11},x_{12},x_{13} )\quad\textnormal{and}\quad   P_2=\P(x_{11},x_{21},x_{31}) $$ that intersect at one point.

The Hilbert scheme $F_2(Q_1)$ of planes contained in $Q_1$ has 2 connected components, and two planes are in different components if and only if they intersect at one point (see \cite[p.~35]{gh}).\ Therefore, $P_1$ and $P_2$ are in two different components of $F_2(Q_1)$.

Consider now our $\P^3\subset Q_0$ disjoint from $ \cM_1 $.\ Its intersection with the vertex $\P^2$ does not meet the conic $\cM_1\cap \P^2$ hence is finite.\ Since the smooth quadric~$Q_1$ contains no 3 dimensional linear spaces, the 
 projection of $\P^3$ from the vertex $\P^2$ to $\P^5$, which is contained in $ Q_1$, is a plane contained in $ Q_1$.\ Therefore, it intersects nontrivially the $P_i$ that is not in the same component of $F_2(Q_1)$.\ Since $P_i$ is contained in $\cM_1$, we obtain that $\P^3$ meets $\cM_1$.

In case (c),   the quadric 
$$Q_0=M^{3,3}=x_{11}x_{22} -x_{12}x_{21}$$ is a cone over a smooth quadric $Q_1 \subset \P^3(x_{11},x_{12},x_{21},x_{22} )$ with
vertex 
$\P^4$.\ Hence, the projection from $\P^4$ to $\P^2$ of any linear subspace  $\P^3\subset Q_0$ is contained in a line on $Q_1$,
say defined by $x_{12}=x_{22}=0$.\ The preimage in $\P^8$ of this line is 
$$\P^6=\P\begin{pmatrix}
 x_{11}& 0&x_{13}\\
 x_{21}&0&x_{23} \\ 
 x_{31}&x_{32}&x_{33}
\end{pmatrix},$$
whose intersection with $\cM_1$ has dimension $3$.\ Therefore,   $\P^3\subset \P^6\subset Q_0$ must meet  $\cM_1$. 
 \end{proof}
To finish the proof of the theorem, we need to prove that (still in the case $k=4$), the image of $r'_K$ is a general hyperplane.\ This follows from the next lemma.\end{proof}

\begin{lemm}\label{lem:GeneralHyper}
    In the above setting, the point $\bp\in \P(\Sym^2\!K)$ defined by the hyperplane $\Im(r'_K)\subset H^0(\P(K), \mathcal{O}(2))$ has maximal rank.
\end{lemm}

  \begin{proof}
  We follow the proof of~\cite[Proposition~2.5]{og} and consider the rational morphism \begin{equation}\label{defphi}
\begin{array}{rcl}
      \Phi\colon \P^8\isom \P(\Hom(U_0, U_\infty))& \dra & |I_{\cM_1}(2)|^\vee\simeq \P^8\\
    M = \begin{pmatrix}
        x_{11} & x_{12} & x_{13} \\
        x_{21} & x_{22} & x_{23} \\
        x_{31} & x_{32} & x_{33} \\
    \end{pmatrix}  & \longmapsto &   \begin{pmatrix}
        M^{1,1} & M^{1,2} & M^{1,3} \\
        M^{2,1} & M^{2,2} & M^{2,3} \\
        M^{3,1} & M^{3,2} & M^{3,3} \\
    \end{pmatrix}
 \end{array}      
\end{equation}
    induced by the linear system $|I_{\cM_1}(2)|\subset H^0(\P^8, \cO(2))$.  

   Using that $\det(M)$ can be expressed as a linear combination of minors of $M$,  one sees that the $(i,j)$-minor of $\Phi(M)$ is equal to 
    $$\Phi(M)^{i,j}= \det(M)\cdot x_{i,j}.$$ 
     It follows that $\Phi$ is a birational involution which
   is regular on $\P^8\smallsetminus \cM_1$ and an isomorphism on~$\P^8\smallsetminus \cM_2$, while $\cM_2\smallsetminus \cM_1$ is contracted onto $\cM_1$.\ 

   We show that $\Phi$ restricted to $\P(K)\subset \P^8\smallsetminus \cM_1$ is an embedding. 
   
   First, we note that the  fibers of $\Phi\colon \cM_2\smallsetminus \cM_1\to \cM_1$ are complements of    quadric surfaces in their span $\P^3$.\ Indeed, the group $  \GL(U_0)\times \GL(U_\infty)$   acts on $I_{\cM_1}(2)$ (see the proof of Theorem~\ref{theorem:ikkr}) and  transitively    on $\cM_1$.\ Therefore, all fibers of $\Phi$ are conjugate under this action.\ Now, the preimage of $\begin{pmatrix}
        1 & 0 & 0 \\
        0 & 0 & 0 \\
        0 & 0 & 0 \\
    \end{pmatrix}$ via $\Phi\colon  \cM_2\smallsetminus \cM_1\to  \cM_1$  is, in the set of   matrices $M$   all of whose minors but $M^{1,1}$ are 0, which is equal to 
    $$
    \P^3=\P\begin{pmatrix}
        0 & 0 & 0 \\
        0 & x_{22} & x_{23} \\
        0 & x_{32} & x_{33} \\
    \end{pmatrix},
    $$
the set of   those   of rank exactly 2.\ It is  defined in this $ \P^3$ by $ x_{22} x_{33}- x_{23} x_{32}\ne 0 $.

The morphism $\Phi\vert_{\P(K)}$ is an embedding outside $\cM_2$ and, at points of $\cM_2\smallsetminus  \cM_1$, its fiber is the  linear space $\P(K)\cap\P^3$, which must be contained in the complement of    a quadric surface, so it is also one reduced point.\ This proves our claim that $\Phi\vert_{\P(K)}$ is an embedding.\ But this restriction can be identified 
  with the composition    $$
       \P(K) \stackrel{\nu_2}{\lhra} \P(\Sym^2\!K) \dra \P(\Im(r'_K)^\vee) ,
    $$
    where $\nu_2$ is the second Veronese embedding and the second map is the projection from the point $\bp$ of $\P(\Sym^2\!K)$ defined by the hyperplane $ \Im(r'_K)\subset \Sym^2\!K^\vee$.\ The fact that it is an embedding means that $\bp$ is not in the second secant variety of $\nu_2(\P(K))$ and leaves only two possibilities:
    \begin{itemize}
     \item[(a)] $\rank(\bp)=3$;
        \item[(b)] $\rank(\bp)=4$ (the hyperplane $\Im(r'_K)$ is general).
   \end{itemize}
   
    We prove that  case (a) cannot happen.\ 
    Since any rank 3 symmetric matrix can be written as the sum of three  rank 1 symmetric matrices, this case   happens if and only if there exist three distinct points $x, y,z \in \P(K)$ such that $\bp = \nu_2(x) + \nu_2(y) + \nu_2(z)$, or equivalently, such that the points $\Phi(x)$, $\Phi(y)$, and $\Phi(z)$ are distinct and lie on a line $\ell$.
    
     We show that in this case, we have $\ell =\Phi(C)$ for some curve $C$ contained in $\P(K)$.\ This leads to a contradiction, since $\Phi$ is a regular embedding on $\P(K)$ which is defined by quadrics, hence the image of any curve in $\P(K)$ has even degree.
    
     We first consider the case $\ell\subset \cM_1$.\ Note that every line in $\cM_1$ is contained in a $\P^2\subset \cM_1$ that is given by matrices with the same $2$-dimensional kernel or the same $1$-dimensional image.\ Up to change of basis and transposition, we can suppose $\P^2 = \P\begin{pmatrix} y_1 & 0 & 0\\
    y_2 & 0 & 0 \\
    y_3 & 0 & 0
    \end{pmatrix}$  and $\ell\subset \P^2$ is defined by the equation $y_3 = 0$.\ Hence, the preimage of this $\P^2$ via $\Phi\vert_{\P^8\smallsetminus \cM_1}$ is the complement in
    $$ 
    \P^5=\P\begin{pmatrix}
        0 & x_{12} & x_{13} \\
        0 & x_{22} & x_{23} \\
        0 & x_{32} & x_{33} \\
    \end{pmatrix} \subset \cM_2
    $$ 
    of the   subvariety $\cM_1\cap \P^5$, and the preimage of  $ \ell $ is    defined by the quadratic equation $x_{12}   x_{23} -
         x_{22}   x_{13}=0$.\ Note that $\P(K)\cap \P^5$ is a linear space that contains the three distinct points $x$, $y$, $z$ but does not meet the codimension $2$ subvariety $\cM_1\cap \P^5$.\ Therefore, $\P(K)\cap \P^5$ is a line and, since it meets the quadric $\Phi^{-1}(\ell)$ in three distinct points, it is contained in it.\ In particular, $\ell$ is the image of this line contained in $\P(K)$, as wanted. 

{We now consider the case $\ell\not\subset \cM_1$.\ In this case, $\Phi^{-1}$ defines a rational morphism on $\ell$  whose image $C$ is an irreducible curve.\ Recall that, outside $\cM_2$, the composition $\Phi^{-1}\circ \Phi$ is the identity.\ Moreover, since $\Phi$ is an automorphism of $\P^8\setminus \cM_2$, we have   inclusions
$$\Phi(\P(K))\cap \cM_2 =  \Phi(\P(K)\cap \cM_2) \subset \Phi(\cM_2\setminus\cM_1) = \cM_1.$$
In particular, $\Phi(\P(K))\cap \cM_2 = \Phi(\P(K))\cap \cM_1$  and $\Phi^{-1}\Phi(w) = w$ for some $w\in\P(K)$ if and only if $\Phi(w)\notin \cM_1$.\ We consider the points $\Phi(x)$, $\Phi(y)$ and $\Phi(z)$ of $\ell$ and distinguish four cases.
\begin{itemize}
\item The three points are all in $\cM_1$.\ This case cannot actually happen: since $\cM_1$ is an intersection of quadrics, each line is either contained in $\cM_1$ or meets it in {at most} $2$ points.
    \item None of the three points are   in $\cM_1$, hence the curve $C$ contains $x$, $y$ and $z$.\ Since the morphism $\Phi^{-1}$ is defined by quadrics, $C $ is a line or a conic that meets the linear space $\P(K)$ in $x$, $y$,   $z$.\ Since all conics are planar, it follows that $C$ is contained in $\P(K)$  and, of course, $\ell =\Phi(C)$. 
    \item Exactly one of these points, say $\Phi(x)$, is contained in $\cM_1$.\ Hence, all quadric polynomials that define $\Phi^{-1}$ vanish on $\Phi(x)$, therefore the curve $C$ is a line that meets the linear space $\P(K)$ in the two distinct point $y$ and $z$, thus it is contained in it.
    \item Exactly two of these points are contained in $\cM_1$.\ In this case, all quadric polynomials that define $\Phi^{-1}$ vanish at two points of $\ell$, hence $\Phi$ is constant on $\ell$, which is absurd. 
\end{itemize}
Therefore, $\ell = \Phi(C)$ for some curve $C$ as wanted, and we obtain a contradiction, as explained above.}
 \end{proof}

\subsection{Local description of $\sZ^k_A$ and proof   of Theorem~\ref{theorem:ikkr}(d)}\label{localde}
    We are now ready to give a local description of the analytic germ of $\sZ^k_A$ at a point $[U_0]$ of $\sZ^\ell_A$ (with $\ell\ge k$) and to prove Theorem~\ref{theorem:ikkr}(d).\ Actually, we will prove, following~\cite[Claim~3.8]{og}, a more general result that will be needed later on.

 Keeping the notation~\eqref{deftu}, we define the   families 
$$
\cZ^{\ge k} \coloneqq\{([A], [U])\in \LG(\bw3V_6)\times \Gr(3, V_6)\mid \dim(A\cap T_U)\ge k\} 
$$ 
of degeneracy loci $\sZ_A^{\ge k}$ and we define $\cZ^{k}$   analogously.

As in~\cite[Claim~3.8]{og}, we study the local structure of $\cZ^{\ge k}$ at a point $([A_0], [U_0])$ of~$\cZ^{\ell}$ for $\ell\ge k$.  

\begin{prop}\label{prop:LocalStructureZ2} 
Let $A_0\subset \bw3V_6$ be a Lagrangian subspace with no decomposable vectors and 
 {let $([A_0],[U_0])$ be a point of $\cZ^{\ell}$, with $\ell\in\{0,\dots,4\}$.\ Set $K\coloneqq A_0\cap T_{U_0}$ and let $\cS_{\le i}\subset\Sym^2\!K^\vee$ be the subscheme of quadratic forms of rank $\le i$.\ 
 
 For any $k\in\{0,\dots,\ell\}$, the analytic germ of the  scheme $\cZ^{\ge k}$ at the  point $([A_0], [U_0])$ is isomorphic to  the product of
  a smooth germ $(M,0)$        and
    the germ   $(\cS_{\le \ell - k},0)$.
    
    Under this isomorphism, the germ of $\sZ^{\ge k}_{A_0}\subset \cZ^{\ge k}$ at $[U_0]$ corresponds to the germ of $M'\times(\Im(r_K)\cap \cS_{\le \ell-k})$ at $(0,0)$, where the morphism $r_K$ is defined in~\eqref{defrk} and $(M',0)\subset (M,0)$ is smooth of dimension $\max(0, 9-\ell(\ell +1)/2)$.
    
    {In particular, for $[A_0]\in \Gamma\setminus \Sigma$, the germ of $\sZ^{\ge2}_{A_0}$ at a point of $\sZ^4_{A_0}$is the germ of $ \Im(r_K)\cap \cS_{\le 2} $ at~$0$.}}
\end{prop}

\begin{proof}
Choose a 3-dimensional subspace $U_\infty$ of $V_6$ complementary to $U_0$ and such that $T_{U_\infty}$ is complementary to $A_0$ (and also to $T_{U_0}$).\ As in the proof of Proposition~\ref{prop:tgtGr36}, we consider   affine open neighborhoods  
 $\cU  \subset\Gr(3, V_6) $ 
of $[U_0]$ and 
 $\cV  \subset\LG(\bw3V_6) $ 
of $[A_0]$,  with identifications $\Hom(U_0, U_\infty)\simeq \cU$ and $\Sym^2\!T_{U_0}^\vee\simeq \cV$.

If  $([A], [U]) \in \cV \times \cU$, we denote by $q_{U}$ and $q_{A}$   the quadratic forms on~$T_{U_0}$ associated with the Lagrangian subspaces $T_U$ and $A $ (which are elements of $\cV$)  respectively.

  Consider the smooth morphism 
$$
 \Psi\colon \cV \times \cU \lra \Sym^2 \!T_{U_0}^\vee 
$$
sending $([A], [U])$ to the quadratic form $q_{U}-q_{A}$.\ It is then easy to check (see for example the beginning of the proof of~\cite[Lemma~2.9]{ikkr}) that the intersection $\cZ^{\ge k} \cap (\cV\times \cU )$ is equal to $ \Psi^{-1}(\Sigma_k)$, where $\Sigma_k\subset \Sym^2\! T_{U_0}^\vee$ is the corank~$k$ degeneracy locus.
 
 {Fix a subspace $J\subset T_{U_0}$ such that $K\oplus J = T_{U_0}$.\ The quadratic form $\Psi([A_0],[U_0])$ is nondegenerate on $J$.\ Consider the vector bundle $\widetilde K \to \cV\times \cU$ whose fiber
    $$\widetilde K_{([A],[ U])} \coloneqq \{\omega \in T_{U_0} \mid  \forall \alpha\in J\quad \widetilde\Psi([A], [U])(\omega, \alpha) = 0 \}\subset T_{U_0}$$ 
 over $([A], [U])$ is the orthogonal complement of $J$ with respect to the quadratic form $\Psi([A], [U])$.

Upon shrinking $\cV$ and $\cU$, we can suppose that $\widetilde K$ is trivial over $\cV\times \cU$ (with fiber $K$) and that $\Psi([A], [U])$ is nondegenerate on $J$ for all $[A]\in \cV$, $[U]\in \cU$.\  
Hence, the corank of $\Psi([A], [U])$ is equal to the corank of the restriction of $\Psi([A], [U])$ to $\widetilde K_{([A], [U])}\simeq K$.\  
Therefore, the degeneracy loci of $\Psi$ are equal to the degeneracy loci of the smooth morphism
$$
\Psi_K  \colon \cV\times \cU \xrightarrow{\ \Psi\ }  \Sym^2\!T_{U_0}^\vee \xrightarrow{\ \pi_K\ } \Sym^2\!K^\vee,
$$
where $\pi_K$ is the restriction morphism.

In particular, $\cZ^{\ge k} \cap (\cV\times \cU)= \Psi_K^{-1}(\cS_{\le \ell-k})$.\ This implies the first claim.\ 
Analogously, the scheme $\sZ^{\ge k}_{A_0}$ is the preimage of $\cS_{\le \ell -k}$ via the morphism
$$
\Psi_K([A_0], \cdot) \colon \cU \lra \Sym^2\!K^\vee.
$$
Note that the morphism $\Psi([A_0], \cdot)\colon\cU \to \Sym^2\!T_{U_0}^\vee$ is induced by the morphism $T$ from Section~\ref{subsec:TgGr36}, whose tangent map induces the isomorphism $\Theta$ from Proposition~\ref{prop:tgtGr36}.\ Therefore, the tangent map to $\Psi_K([A_0], \cdot)$ at $[U_0]$ is the morphism~$r_K$ of~\eqref{defrk}.\ Hence, analytically locally around $[U_0]\in \sZ^{\ell}_{A_0}$, the image of the subscheme $\sZ^{\ge k}_{A_0}\subset \cZ^{\ge k}$ can be identified with $\Im(r_K)\cap \cS_{\le \ell-k}\subset \cS_{\le \ell-k}$.

This proves that the germ of $\sZ^{\ge k}_{A_0} $ at $[U_0]$ is the product of the germ  $\Im(r_K)\cap \cS_{\le \ell-k}$ at $0$ with a smooth germ  whose dimension can be computed using the fact, proved in   Proposition~\ref{prop23}, that~$r_K$ has maximal rank. } 
\end{proof}

The following corollary proves in particular Item~(d) of Theorem~\ref{theorem:ikkr}.

\begin{coro}  Let $A\subset \bw3V_6$ be a Lagrangian subspace with no decomposable vectors and  let $[U_0]\in \sZ^k_A$.\ The scheme $\sZ^k_A$ is smooth at $[U_0]$, of dimension $9-\frac{k(k+1)}2$ when $k\in\{0,1,2,3\}$, and of dimension $0$ when $k=4$.
\end{coro}

\begin{proof}
    This is just the case $k=\ell$ of the proposition.
\end{proof}

 \section{Divisors in the Lagrangian Grassmannian $\LG(\bw3V_6)$}\label{sec2}

\subsection{The divisors $\Gamma$, $\Sigma$, and $\Delta$ in $\LG(\bw3V_6)$} \label{sec31}

 O'Grady introduced in \cite[Section~1]{og2} the divisors
 $$
 \begin{array}{rcl}
       \Sigma&\coloneqq& \{[A] \in \LG(\bw3V_6)\mid A \mbox{ contains a decomposable vector}\}, \\
     \Delta & \coloneqq& \{[A] \in \LG(\bw3V_6)\mid \sY_A^{\ge 3} \ne\vide\} 
 \end{array}
 $$
in $\LG(\bw3V_6)$ and proved in \cite[Proposition~3.1]{og2} and \cite[Proposition~2.2]{og} that they are irreducible and distinct.

 In \cite[Lemmas~3.6 and~3.7]{ikkr}, it is proven that 
 \begin{equation}\label{eqn:DefGamma}
     \Gamma\coloneqq \{[A] \in \LG(\bw3V_6)\mid \sZ_A^{\ge 4} \ne\vide\}
 \end{equation}
is a divisor in $ \LG(\bw3V_6)$ which is distinct from O'Grady's irreducible divisor $\Delta$.\ We prove here a stronger version this result.%\footnote{In the notation used below, the proof that $\widetilde\Gamma$ has dimension $55$ is the same, but the assertion in the proof of~\cite[Lemma~3.6]{ikkr} that the projection $\widetilde\Gamma\to\Gamma$ is finite is not proved there (and is likely to be false over $\Gamma\cap\Sigma$).}

\begin{prop}\label{prop:gamma}
    The loci  $\Gamma$, $\Sigma$, and $\Delta$ are mutually distinct irreducible divisors  in $\LG(\bw3V_6) $. 
\end{prop}

\begin{proof}
  Consider   the closed subvariety  
    $$\widetilde\Gamma  \coloneqq \{([A], [U], [B])\in\LG(\bw3V_6)\times \Gr(3, V_6)\times \Gr(4, \bw3V_6)\mid B\subset A\cap T_U \},$$
   where $T_U$ was defined in~\eqref{deftu}, and the chain of morphisms
        $$
    \begin{array}{rcccl}
        \widetilde\Gamma  &\xrightarrow{\ q\ } & \Gr(4, T_\cU) &\lra &\Gr(3, V_6)\\
        ([A], [U], [B]) &\longmapsto &([U],[B]) &\longmapsto &[U].
    \end{array}$$
     where $\Gr(4,T_\cU)$ is the $\Gr(4,10)$-bundle over $\Gr(3, V_6)$ whose fiber at $[U]$ is equal to $\Gr(4,T_{U})$. 
     
    The fiber of $q$ over $B\subset T_U$ is isomorphic to the set of Lagrangian subspaces $A$ such that $B\subset A$.\ Since $B\subset T_U$ is totally isotropic, this fiber is isomorphic to the Lagrangian Grassmannian $\LG(B^\perp/B)$, smooth irreducible of dimension 21.

    Since $\Gr(4, T_\cU)$ is smooth irreducible of dimension $33$, it follows that $\widetilde\Gamma $ is also smooth irreducible, of dimension $54$.\ Projecting $\widetilde\Gamma $ to $\LG(\bw3V_6)$, we obtain that its image $\Gamma$ is also irreducible.% of dimension $\le 54$, hence distinct from the $55$-dimensional $\LG(\bw3V_6)$.

  % By \cite[Remark~5.29]{dk1}, $\Gamma$ is not contained in $\Sigma $.\ By Theorem~\ref{theorem:ikkr}(c), for any $[A]\in \Gamma\smallsetminus \Sigma$, the scheme $\sZ^4_A$ is finite and $\sZ^5_A$ is empty.\ So, for any $ ([A], [U], [B])\in\widetilde\Gamma$ over it, one has $[U]\in \sZ^4_A$ and $B=A\cap T_U$.\ In particular, the projection $\widetilde\Gamma\to\Gamma$ is finite over $\Gamma\smallsetminus \Sigma$, hence $\Gamma$ has dimension exactly 54: it is an irreducible divisor in $\LG(\bw3V_6)$.
    
For the fact that $\Gamma$, $\Sigma$, and $\Delta$ are mutually distinct, we refer to Proposition~\ref{prop:imagediv} (proved independently of the rest of the paper, of course).
\end{proof}

\subsection{The divisor $\Gamma$: parameter count.}\label{sec:ParamCount}

Following \cite[Section~2.1]{og}, we show that for $[A]\in \Gamma$ general, $\sZ^{\ge 4}_A$ consists of a single point.

\begin{lemm}\label{lemma:parameterCountGamma+}
   The subvariety 
    $$
    \Gamma_+ = \{[A]\in \LG(\bw3V_6)\mid |\sZ_A^{\ge 4}|> 1\}
    $$
    of $\LG(\bw3V_6)$ has codimension at least $2$.\ Therefore, for $[A]\in \Gamma$ general, the scheme $\sZ^{\ge 4}_A$ consists of a single smooth point, which is the only  singular point of the scheme $\widetilde \sZ^{\ge 2}_A$, defined in Equation~\eqref{eqn:EPWcube}. 
\end{lemm}

\begin{proof}
Consider the locally closed subvariety
\begin{align*}
    %\widetilde\Gamma_+ &= \{([A], [U_1], [U_2]) \mid [U_1]\neq [U_2],\quad \dim(A\cap T_{U_i})\ge 4 \text{ for } i =1,2\},\\    
    \widetilde\Gamma_+ &= \{([A], [U_1], [U_2]) \mid [U_1]\neq [U_2],\ \dim(A\cap T_{U_1})= \dim(A\cap T_{U_2})=4 \}
\end{align*}
of $\LG(\bw3V_6)\times \Gr(3,V_6)^2$.\ 
It follows from Theorem~\ref{theorem:ikkr}(e) that 
 $\Gamma_+$ is contained in the union of $\Gamma\cap\Sigma$ and  the image of the first projection $\pi_+\colon  \widetilde\Gamma_+\to \LG(\bw3V_6)$.\ Since $\Gamma$ and $\Sigma$ are distinct irreducible divisors by Proposition~\ref{prop:gamma}, it is therefore enough to show that $\widetilde\Gamma_+$ has dimension at most $53$, .\  

We consider the morphism
\begin{align*}
    \eta\colon \widetilde\Gamma_+ &\lra \Gr(4, \bw3V_6)^2\times  \Gr(3,V_6)^2\\
    ([A], [U_1], [U_2]) &\longmapsto  ([A\cap T_{U_1}], [A\cap T_{U_2}], [U_1], [U_2]).
\end{align*} 
Since $A$ is a Lagrangian subspace, the image of $\eta$ is contained in
$$
    N = \left\{([K_1], [K_2], [U_1], [U_2]) \;\middle| \; 
    \begin{aligned}
        &[U_1]\neq [U_2], &[K_1]\in \Gr(4, T_{U_1}), \quad [K_2]\in \Gr(4, T_{U_2}),\\
        &K_1\perp K_2, &K_1\cap T_{U_2} = K_1\cap K_2 = K_2\cap T_{U_1}
    \end{aligned}\\
       \right\}.
$$
%
%  The inclusion $\Im(\eta)\subset N$ follows from the fact that $A$ is a Lagrangian subspace.} 
Observe that  the fiber of an element $([K_1], [K_2], [U_1], [U_2])$ of $N$ is isomorphic to 
$$\{[A]\in \LG(\bw3V_6)\mid K_1 = A\cap T_{U_1},\ K_2 = A\cap T_{U_2} \},$$
 which, since $K_1+K_2$ is totally isotropic, is contained in
 %an open\footnote{\rem{Possibly empty (this would not matter for the proof, we would just have that the dimension is $\le \frac{(2+k)(3+k)}2$)? If $A\supset K_1+K_2$, then $K_1 \subset A\cap T_{U_1}$ and $K_2 \subset A\cap T_{U_2}$, but would there be equality for some $A$?}  \Fr{I think yes.\ Indeed, choose $B$ such that $\LG((K_1+K_2)^\perp/(K_1+K_2))\cap \cU_B$ is not empty, and let $p_2$ and $p_B$ be the projections from $\bw3V_6$ to $T_2$ and $B$ respectively.\  We set $K_1 = (K_1\cap K_2)\oplus K_1'$, and denote by $q_{K_1}$ the morphism $p_2(K_1')\to p_B(K_1'))$ whose graph is $K_1'$.\ Note that $q_{K_1}$ is injective, since $K_1'\cap K_2 = 0$.\ The condition $K_1+K_2\subset A$ is equivalent to look for morphism $q\in \Sym^2(T_2^\vee)$ such that $K_2\subset \ker(q)$ and $q_{|_{K_1'}} = q_{K_1}$ (the two conditions are compatible, i.e.\  $p_2(K_1')\cap K_2 = 0$, since we are supposing that $\LG((K_1+K_2)^\perp/(K_1+K_2))\cap \cU_B$ is not empty).\ Therefore, if we write $T_{U_2} = K_2\oplus p_2(K_1') \oplus T$, we have $$\LG((K_1+K_2)^\perp/(K_1+K_2))\cap \cU_B = \Sym^2 T^\vee,$$ where given $q'\in \Sym^2 T$, the Lagrangian associated $A_{q'}$ to $q$ is $K_1+K_2\oplus \Gamma(q')$.\ Finally note that $A_{q'}\cap T_2 = K_2\oplus \ker(q')$ and $A_{q'}\cap T_1 = K_1 \oplus \ker(q'-q_1')$, where $q_1': T'\to (T')^\vee$ is such that $T_1 = K_1\oplus \Gamma(q_1')$. }} subset of  
 $\LG((K_1+K_2)^\perp/(K_1+K_2))$.\  
 Therefore, the fiber $\eta^{-1}([K_1], [K_2], [U_1], [U_2])$ has dimension at most $ \frac{(2+k)(3+k)}2$, where $k = \dim(K_1\cap K_2)$.

% In particular, if we let $N_k$ be the subspace of $N$ where $\dim(K_1\cap K_2)=k$, we obtain that the dimension of $\widetilde\Gamma_+$ is the maximum of $\dim(\eta^{-1}(N_k))$ for $0\le k\le 4$, namely
% \begin{equation}
%     \dim(\widetilde\Gamma_+) = \max_{0\le k\le 4} \frac{(2+k)(3+k)}2 + \dim(N_{k}).
% \end{equation}

We stratify $N$ according to $k = \dim(K_1\cap K_2)$ and $\ell = \dim(U_1\cap U_2)$, and denote by 
$$N_{k, \ell} = \{([K_1], [K_2], [U_1], [U_2])\in N\mid\ell = \dim(U_1\cap U_2), \ k = \dim(K_1\cap K_2)\}$$
the stratum corresponding to $k$ and $\ell$.
Clearly we have the inclusion $K_1\cap K_2 \subset T_{U_1}\cap T_{U_2}$, where the dimension $d(\ell)$ of $T_{U_1}\cap T_{U_2}$  is $0$ for $\ell = 0$, and $2(\ell+1)$ for $\ell\in\{1,2\}$.\ Moreover, the morphism $\eta$ has equidimensional fibers over $N_{k, \ell}$.\ Thus, the dimension of $\widetilde\Gamma_+$ is the maximum of $\dim(\eta^{-1}(N_{k, \ell}))$ for $0\le k\le d(\ell)$ and $0\le \ell\le 2$, namely
\begin{equation}
    \dim(\widetilde\Gamma_+) = \max_{0\le k\le d(\ell)} \frac{(2+k)(3+k)}2 + \dim(N_{k,\ell}).
\end{equation}

We   count parameters to determine the dimension of $N_{k,\ell}$.\ The space of pairs $([U_1], [U_2])\in  \Gr(3,V_6)^2$ whose intersection has dimension $ \ell$ is a Schubert cell of dimension $18-\ell^2$.\ Having fixed such a pair $([U_1], [U_2])$, the quadruple $([K_1], [K_2], [U_1], [U_2])$ is in $N_{k, \ell}$ if and only if $K_1\cap T_{U_2}$ has dimension~$k$ and 
$$
K_1\cap T_{U_2} \subset K_2\subset K_1^\perp \cap T_{U_2}\ , \quad K_2\cap T_{U_1} = K_1\cap T_{U_2}.
$$
 Hence, given $[K] = [K_1\cap T_{U_2}]$ in $\Gr(k, T_{U_1}\cap T_{U_2})$, and $[K_1]$ in $\Gr(4-k, T_{U_1}/K)$, the point~$[K_2]$ lies in $\Gr(4-k, ( K_1^\perp \cap T_{U_2})/K)$.\  Note that, since $K_1\cap T_{U_2}$ has dimension $k$,  the space $K_1^\perp \cap T_{U_2}$ has dimension $6+k$.\  Putting all together, we obtain \begin{align*}
    \dim(N_{k, \ell}) &\le 18 - \ell^2 + k (d(\ell)-k) + 6(4-k) + (4-k)(2+k)\\
    &=50 - 2(k + 2)k + d(\ell) k - \ell^2.
\end{align*}  
Hence, the dimension of $\eta^{-1}(N_{k,\ell})$ is at most $53 - \ell^2 - \frac{3}2 k ( k - \frac{2d(\ell) - 3}3)$, whose maximum for $0\le k\le \max\{d(\ell),4\}$ and $0\le \ell \le 2$ is 53.\ Therefore, $\widetilde \Gamma_+$ has dimension at most 53, as wanted. 
\end{proof}
\section{The double cover $\widetilde \sZ^{\ge 2}_A\to \sZ^{\ge2}_A$}\label{se4}

In \cite[Theorem~5.7(2)]{dk2}, the authors construct, for each   $[A]\in \LG(\bw3V_6) $ with no decomposable vectors, a canonical double cover
$$
g\colon \widetilde \sZ^{\ge 2}_A \lra \sZ^{\ge 2}_A
$$
branched over $\sZ^{\ge 3}_A$.\ The variety $\widetilde \sZ^{\ge 2}_A$ is integral, normal, and smooth away from $g^{-1}(\sZ^{4}_A)$. 

When moreover $[A]\notin \Gamma$, this double cover coincides with the   EPW cube of \cite{ikkr} (see \cite[Lemma~5.8]{dk2}); it is therefore a   smooth \hk\ variety of $\KKK^{[3]}$-type with a polarization $H_A$ of square $4$ and divisibility $2$.\ In this section, we study the singular sixfold $\widetilde \sZ^{\ge 2}_A$ when $[A]\in \Gamma\setminus \Sigma$ and describe its   blowup    along its singular locus  and some  smooth \hk\ resolutions.
%\footnote{\Fr{Je me rappelle  que   $\cZ^4\subset \cZ^{\ge 2}$ a la bonne codimension, mais je ne vois pas où on l'utilise...}} 

We will study these double covers and  the desingularization of  $\widetilde\sZ^{\ge 2}_A$  in families.\ So we start with the families
$$
\cZ^{\ge k} \lra  \LG(\bw3V_6) 
$$ 
defined in Section~\ref{localde}.\ 

Let $[A_0]\in \Gamma\setminus \Sigma$ and let $\cV\subset \LG(\bw3V_6)\setminus \Sigma$ be a (sufficiently small classical) open neighborhood of $[A_0]$.\ We set $\cZ_{\cV}^{\ge k} \coloneqq \cZ^{\ge k} \cap (\cV\times \Gr(3, V_6))$.\ As in the case of   EPW sextics (see the proof of \cite[Theorem~5.2]{dk2}), there exists a unique double cover 
        $$f_\cV \colon \widetilde \cZ_\cV \lra  \cZ^{\ge 2}_\cV$$
branched along $\cZ^{\ge 3}_\cV$, 
whose restriction over $\sZ^{\ge 2}_A$ is the morphism $g\colon \widetilde \sZ^{\ge 2}_A \to \sZ^{\ge 2}_A$.\

Given  $[U_0]\in  \sZ_A^{\ge 4}$, we constructed 
 in the proof of Proposition~\ref{prop:LocalStructureZ2} a smooth morphism
 \begin{equation*}\label{eqn:morphToS_2}  
 \Psi_K \colon \cZ^{\ge 2}\cap (\cV\times\cU) \lra \cS_{\le 2},
\end{equation*}
where $\cU$ is a  neighborhood of $ [U_0]$ in $  \Gr(3, V_6)$, and $\cS_{\le 2}$ is the scheme of $(4\times 4)$-matrices of rank~$\le 2$.\ We prove that the double cover~$f_\cV$ is the pullback by $\Psi_K$ of a certain  ``universal double cover'' of $\widetilde\cS_{\le 2}\to \cS_{\le 2}$  studied in Section~\ref{sec:CanonicalDoubleCover}.\ In particular, we obtain a description of the blowup of $\widetilde \cZ^{\ge 2}$ along its singular locus~$f_\cV^{-1}(\cZ^{\ge 4}_\cV)$, from which we deduce a description of the blowup of $\widetilde \sZ^{\ge 2}_A$.

 \subsection{Local model of double cover}\label{sec:CanonicalDoubleCover}
In this section, we consider a ``canonical example'' of double cover.\ Let $V$ be a vector space of dimension $n$ and consider the universal quadratic form
$$
q\colon  V\otimes \cO_{\Sym^2\!V^\vee} \lra  \!V^\vee\otimes \cO_{\Sym^2\!V^\vee}.
$$
 Denote by $\cS_{\le i}\subset \Sym^2\!V^\vee$ the  rank $i$ degeneracy locus of $q$,
 %and  by $\cS_{\le i}^0$ its smooth locus $\cS_{\le i}\setminus \cS_{i-1}$, 
 so that \begin{itemize}
    \item $\cS_{\le 2}$ is the affine cone over the secant variety to the second Veronese embedding $\P(V^\vee)\subset \P(\Sym^2\!V^\vee)$,
    \item $\cS_{\le 1}$ is the affine cone over the same embedding $\P(V^\vee)\subset \P(\Sym^2\!V^\vee)$,
    \item $\cS_{\le 0}=\{0\}$.
\end{itemize} 
In this situation, \cite[Theorem~3.1]{dk2} provides a canonical double cover $\widetilde f\colon \widetilde \cS_{\le 2} \to \cS_{\le 2}$.
%, where $\widetilde \cS_{\le 2}$ is normal and characterized (up to isomorphism) by the following two properties:
   % \begin{enumerate}[label = (\alph*)]
       % \item $\widetilde f_* \cO_{ \widetilde\cS_{\le 2}} \simeq \cO_{ \cS_{\le 2}} \oplus R_k$, where $R_k$ is the reflexive sheaf on $\cS_{\le 2}$ corresponding to the invertible sheaf $\det(\Im(q\vert_{\cS_{\le 2}^0}))$ on $\cS_{\le 2}^0$ (see \cite[Equation~(3)]{dk2} and the proof of \cite[Theorem~3.1]{dk2});
       % \item $\widetilde f$ is étale over $\cS_{\le 2}^0$.
   % \end{enumerate}
    
\begin{lemm}\label{lemma:CanonicDoubleCover}
    The affine variety $(V^\vee\otimes V^\vee)_1 = \{\mu \in V^\vee\otimes V^\vee \mid \rk(\mu)\le 1\}$ is the cone over the Segre variety $\P(V^\vee)\times \P(V^\vee)$.\ The map 
    \begin{align*}
        g_2 \colon (V^\vee\otimes V^\vee)_1 & \lra \cS_{\le 2}\\
                    \mu \quad &\longmapsto \mu+ {\iota(\mu)} ,
    \end{align*}
   where $\iota$ is the involution that exchanges the two factors,   is the canonical double cover $\widetilde f\colon \widetilde \cS_{\le 2} \to \cS_{\le 2}$.\ In particular, the unique singular point of the variety $\widetilde \cS_{\le 2}$ is $0$.
\end{lemm}

\begin{proof}
    We follow the proof of \cite[Lemma~3.4]{dk2}.\ Note that $f$ is a double cover because it is the quotient for the action of the involution $\iota$ on $(V^\vee\otimes V^\vee)_1$; it is moreover \'etale over~$\cS_{\le 2}^0$.\ Thus, since $(V^\vee\otimes V^\vee)_1$ is normal, in order to show that $g_2$ and $\widetilde f$ are isomorphic, it is enough to show that their associated  reflexive sheaves are the same.
    %f associated to $g_2$ (in the sense of (a)) is the sheaf $R_k$.

    Let $R$ be the ring of functions of the affine variety $(V^\vee\otimes V^\vee)_1$.\ If we denote by $S^\bullet(V\otimes V)$ the symmetric algebra of $V\otimes V$,   we can identify $R$ with the quotient $S^\bullet(V\otimes V)/I$, where $I$ is the ideal generated by $\{(v\otimes w)(v'\otimes w') - (v\otimes w')(v'\otimes w) \mid v,w,v',w'\in V\}$. {Note that there is an isomorphism 
    \begin{equation}\label{eqn:coordRing}
       R\simeq \bigoplus_{i=0}^\infty S^i(V)\otimes S^i(V)
    \end{equation}
    that sends the class of a monomial $(v_1\otimes w_1)\cdots (v_j\otimes w_j)\in S^j(V\otimes V)$ to $(v_1\cdots v_j)\otimes (w_1\cdots w_j)$.}

    Given a subset~$P$ of $S^\bullet(V\otimes V)$, we denote by $[P]$ its image in $^\bullet S(V\otimes V)/I$.

 The action of $\Z/2\Z$ on $(V^\vee\otimes V^\vee)_1$ induces an action on $R$  given by the involution $\iota$ that sends $[v\otimes w]$ to $[w\otimes v]$.\ We have a decomposition 
    \begin{equation}\label{eqn:RefSheafDoubleCover}
        R = R_+\oplus R_-,
    \end{equation}
    where $R_+$ is the invariant part and $R_-$ is the antiinvariant part with respect to this action.\ 
Under the isomorphism \eqref{eqn:coordRing}, their graded part are given by  
    $$
        R_+ = \bigoplus_{i=0}^\infty S^2(S^iV)\quad \text{and}\quad
        R_- = \bigoplus_{i=0}^\infty \bw2(S^iV).
    $$
     Note that $R_-$ is the vector space generated by $[\alpha - \iota(\alpha)]$, for all monomials $\alpha$ in $S^\bullet(V\otimes V)$.\ Any monomial $\alpha$ in $S^j(V\otimes V)$ can be written as $\alpha = (v\otimes w)\beta$, where $\beta\in S^{j-1}(V\otimes V)$.\ Thus, by induction on $j$ using the   equality
$$
     2(\alpha - \iota(\alpha)) = (v\otimes w - w\otimes v)(\beta + \iota(\beta)) + (v\otimes w + w\otimes v)(\beta - \iota(\beta)),
 $$
 we see that $R_-$ is generated by $[\bw2V]$ as an $R_+$-module.%\footnote{\Fr{Maybe simplify using \eqref{eqn:coordRing}?}}.
    
    %The ring $R_+$ can be  identified with  $\overline{S(\Sym^2 V)}$, while $R_-$ is the $R_+$-module generated by $\overline{\bw^2V}$.\\
     Then $\cS_{\le 2} = \Spec(R_+)$,
      %and, under the identification of $\cS_{\le 2}$-modules with $R_+$-modules, 
      the direct sum \eqref{eqn:RefSheafDoubleCover} provides the decomposition of $g_{2,*}\cO_{(V^\vee\otimes V^\vee)_1}$ into invariant and antiinvariant parts, and the reflexive sheaf associated to $g_2$ is the sheaf associated with the module $R_-$. 

     The restriction of $q$ to $\cS_{\le 2}$ is induced by the morphism 
    $$
    \begin{array}{rcl}
      q\colon  V\otimes R_+  &\lra &V^\vee \otimes R_+ = \Hom(V\otimes R_+, R_+)\\
        v\otimes 1 &\longmapsto &(w\otimes 1\mapsto \frac{1}{2}[v\otimes w + w\otimes v])
    \end{array}
    $$
   of   free $R_+$-modules.\ Set $\cS_{2}\coloneqq \cS_{\le 2}\setminus \cS_{\le 1}$.\ We show that the sheaf $  \det(\Im(q\vert_{\cS_{2}}))$ (which is, by the construction of $\widetilde f$ in \cite[Theorem~3.1]{dk2}, the reflexive sheaf associated to $\widetilde f$) is the sheaf associated to the $R_+$-module $R_-$.\ Indeed, over $\cS_{2}$, the image of $q$ is locally free of rank $2$, thus 
    $$
        \det(\Im(q\vert_{\cS_2})) = \bw2\Im( q\vert_{\cS_2}) = \Im\bigl(\bw2q\vert_{\cS_2}\bigr).
    $$
  The morphism $$\bw2q \colon \bw2V\otimes R_+ \lra \Hom(\bw2V\otimes R_+, R_+)$$ sends an element $(v_1\wedge v_2)\otimes 1\in \bw2V\otimes R_+$ to the dual form defined by
  $$
  (w_1\wedge w_2)\otimes 1 \longmapsto q(v_1, w_1)q(v_2, w_2) - q(v_1, w_2)q(v_2, w_1).
  $$ 
  By developing the latter expression, we obtain that $\bw2q$ factors (up to a constant) as 
  $$
    \begin{array}{rcccl}
        \bw2 V\otimes R_+  &\lra & R_-&\lra&\Hom(\bw2V\otimes R_+, R_+)\\%= (\bw2V\otimes R_+)^\vee\\
        (v_1\wedge v_2)\otimes 1 &\longmapsto& v_1\wedge v_2&\longmapsto &\bigl((w_1\wedge w_2)\otimes 1\mapsto -(w_1\wedge w_2)(v_1\wedge v_2)\bigr),
    \end{array}
    $$
    where we identify  $\bw2V$ with $(R_-)_1$.
    %we denote by $a \wedge b$ the product $a\otimes b - b\otimes a$.
    
    Note that the first morphism is surjective, since $R_-$ is generated by $\bw2V$ as an $R_+$-module, while the second morphism is injective.\ Thus $\Im\bigl(\bw2q\vert_{\cS_2}\bigr)$ is the sheaf associated to~$R_-$, as wanted.
\end{proof}

%\begin{rema}\label{rema:singCoverS2}
    %\Fr{The affine variety $(V^\vee\otimes V^\vee)_1$ is the cone over the Segre variety $\P(V^\vee)\times \P(V^\vee)$. Hence, the variety $\widetilde \cS_{\le 2}$ has a unique singular point at $0$.}
    %(indeed every rank 1 morphism $\mu\colon  V\to V^\vee$ can be written as $\mu = \ell\cdot \ell'$ for $\ell, \ell'\in V^\vee$). 
%\end{rema}

 Let $\rho\colon \widetilde T \to \widetilde\cS_{\le 2}$ be the blowup  of the unique singular point $0$ of $\widetilde \cS_{\le 2}$ and let $H\subset \Sym^2\!V^\vee$ be a general hyperplane.  

\begin{prop}\label{prop:BlowUpCoverS2}
    {The exceptional divisor $E$ of $\rho$ is isomorphic to $\P(V^\vee)\times\P(V^\vee)$, and the normal bundle of $E$ in $\widetilde T$ is isomorphic to $\cO_{\P(V^\vee)\times \P(V^\vee)}(-1, -1)$.
    
    Moreover, the intersection of  the strict transform of $g_2^{-1}(H)$ with $E$ is a general hyperplane section of $E$.}
\end{prop}

\begin{proof}
    {The description in Lemma~\ref{lemma:CanonicDoubleCover} of $\widetilde\cS_{\le 2}$ as the cone over the Segre variety $\P(V^\vee)\times \P(V^\vee)$  implies the description of the exceptional divisor $E$ and its normal bundle.
    
    If we choose coordinates $(x_i)_{1\le i  \le n}$ on $ V^\vee$, which induce  coordinates $(s_{ij})_{1\le i\le j \le n}$ on $\Sym^2V^\vee$, then $g_2^{-1}(H)$ can be identified with the cone over the preimage of the hyperplane $s_{11}+\cdots+s_{nn}=0$ via the morphism
    \begin{align*}
        \P(V^\vee)\times  \P(V^\vee) &\lra  \P(\cS_{\le 2})\subset \P(\Sym^2V^\vee)  \\
        ([(x_i)], [(y_j)]) &\longmapsto [(x_iy_j+y_jx_i)]_{1\le i\le j \le n},
    \end{align*}
    which is given by the equation $x_1y_1+\cdots+x_ny_n=0$.\ Thus, the strict transform of $g_2^{-1}(H)$ intersects the exceptional divisor $\P(V^\vee)\times  \P(V^\vee)$ as a general hyperplane section.}
\end{proof}

%\begin{rema}\label{rema:normalBundle}
  % \Fr{ Let $p$ and $q$ be the two projections $E\to \P(V^\vee)$.\ The normal bundle $N_E$ to $E\subset \widetilde T$ is isomorphic to $p^*\cO_{\P(V^\vee)}(-1)\otimes q^*\cO_{\P(V^\vee)}(-1)$, and $N_E$ restricted to the fibers of $q$ (or of $p$) is isomorphic to $\cO_{\P(V^\vee)}(-1)$.} 
%\end{rema}

\subsection{Singularities of $\widetilde \sZ^{\ge 2}_A$ for $[A]$ in $\Gamma\smallsetminus\Sigma$ and simultaneous resolution}\label{sect42}

 We now go back to the double cover
  $$  \widetilde \cZ_\cV \xrightarrow{\ f_\cV\ }  \cZ^{\ge 2}_\cV\lra\cV$$
  branched along $\cZ^{\ge 3}_\cV$ mentioned at the beginning of Section~\ref{se4} (recall that $\cV\subset \LG(\bw3V_6)\setminus \Sigma$ is a  sufficiently small classical  open neighborhood of a given point $[A_0]\in \Gamma\smallsetminus\Sigma$).\  The next theorem explains the geometry of $  \widetilde \cZ_\cV $; note that the situation is completely analogous (in dimension 6 instead of 4) to  the situation for double EPW sextics described in \cite[Claim~3.8]{og}. 

\begin{theo}\label{theo:blowUpFamily}
    With notation  as previously, we have the following.
    \begin{itemize}
        \item[\rm(a)]   The singular locus of $\widetilde \cZ_\cV $ is $f_\cV^{-1}(\cZ^{\ge 4}_\cV)$ and it is finite over $\cV$.
         \item[\rm(b)] The blowup  $$\rho_\cV \colon \widetilde \cX_\cV \lra \widetilde \cZ_\cV$$
      of    $f_\cV^{-1}(\cZ^{\ge 4}_\cV)$ in $\widetilde \cZ_\cV$ is smooth.\ 
       \item[\rm(c)] The morphism $f_\cV\circ \rho_\cV$ restricted to the exceptional divisor $E_\cV$ of $\rho_\cV$ induces a locally trivial fibration
        $$
            E_\cV \lra \cZ^{\ge 4}_\cV 
        $$
        %with fiber $\P((A\cap T_{U})^\vee)\times \P((A\cap T_{U})^\vee)=\P^3\times \P^3$ over $([A], [U])$, 
        with fibers $ \P^3\times \P^3$, and the   restriction of the normal bundle to $E_\cV$ in $\widetilde \cX_\cV$ to each  fiber  is $\cO(-1,-1)$.
          \item[\rm(d)] 
        For any $[A]\in\cV\cap\Gamma$, the strict transform $\widetilde\sX_{A}$ %\footnote{\Fr{is this smooth? A posteriori oui, mais je ne sais pas si on a besoin de la lissité (mais on utilise que est normal et locally factorial)\\
       % Anyways it should be smooth from Sasha's answer here https://math.stackexchange.com/questions/1943978/blow-up-of-cone-over-quadric-surface}} 
       of $\widetilde\sZ^{\ge 2}_A\subset \widetilde \cZ_\cV$ is smooth and intersects~$E_\cV$ in a general hyperplane section of $\P^3\times \P^3$. 
    \end{itemize}
\end{theo}  

\begin{proof}
For any $([A],[U])\not\in \cZ^{\ge 4}_\cV$, consider an open neighborhood $\cV'\times \cU\subset \cV\times \Gr(3, V_6)$ of $([A], [U])$ such that  $\cZ^{\ge 4}\cap (\cV'\times \cU)=\varnothing$.\ By uniqueness of the double cover constructed in \cite{dk2}, the variety $\widetilde \cZ_\cV\cap (\cV'\times\cU)$ is the double cover of $\cZ^{\ge 2}\cap (\cV'\times \cU)$, hence is smooth by  \cite[Corollary~4.8]{dk2}.\  
Therefore, the singular locus of $\widetilde \cZ_\cV$ is contained in $f^{-1}(\cZ^{\ge 4}_\cV)$. 

 Let us work now around a point $([A_0], [U_0])$ of $\cZ^{\ge 4}_\cV$.\  By Proposition~\ref{prop:LocalStructureZ2} (with $k =2$ and $\ell = 4$),  the analytic germ of $\cZ^{\ge 2}_\cV$ at that point is the product of a smooth germ   and the germ of $\cS_{\le 2}$ at $0$, where $\cS_{\le 2}\subset\Sym^2(A_0\cap T_{U_0})^\vee$ is the subscheme of quadratic forms of rank $ \le 2$.\ In particular, locally, we have a diagram  
$$
    \begin{tikzcd}
    \widetilde \cZ_\cV \ar[r] \ar[d, "f_\cV"'] & \widetilde \cS_{\le 2} \ar[d, "g_2"]\\
    \cZ^{\ge 2}_\cV  \ar[r, " \Psi_K"] &\cS_{\le 2},
\end{tikzcd}
    $$
 where $\Psi_K$ is smooth and $g_2$ is the canonical double cover of $\cS_{\le 2}$ (see Lemma~\ref{lemma:CanonicDoubleCover}).\ Moreover, the variety $\widetilde \sZ^{\ge 2}_{A}$ is the product of a smooth germ with $ g_2^{-1}(\Im(r_K)\cap \cS_{\le 2})$, where $\Im(r_K)\subset \Sym^2(A_0\cap T_{U_0})^\vee$ is a general hyperplane (Lemma~\ref{lem:GeneralHyper}). 

Therefore, it is enough to show (a), (b), (c), and   (d) for  the scheme $ \widetilde \cS_{\le 2} $  described in Lemma~\ref{lemma:CanonicDoubleCover}) and the preimage of a general hyperplane section of $\cS_{\le 2}$.\ 
This was done in Section~\ref{sec:CanonicalDoubleCover}. 
\end{proof}

Given $[A_0]\in \Gamma\setminus \Sigma$ with $\sZ^{\ge 4}_{A_0}=\{z_1,\dots,z_r\}$, we have shown that the blowup of $\widetilde\sZ^{\ge2}_{A_0}$ along it singular locus $\sZ^{\ge 4}_{A_0}$ is 
$$
\tau\colon \widetilde \sX_{A_0} \lra \widetilde\sZ^{\ge2}_{A_0},
$$
with exceptional divisor $E=E_1\sqcup \cdots\sqcup E_r$, the disjoint union of $r$ copies of the incidence variety $I\subset \P^3\times (\P^3)^\vee$ (this is   a general hyperplane section), and $\cO(E)\vert_{E_i}=\cO_{E_i}(-1, -1) $.\\

The situation is therefore exactly as in \cite[Section~3.2]{og}, which we follow without repeating the proofs.\ For each $i\in\{1,\dots, r\}$,   choose a projection $E_i\to \P^3$.\ We call that a choice of projections $\eps$
 for $\sZ^{\ge 2}_{A_0}$.\ 

 \begin{prop}\label{prop:contraction}
        Let $[A_0]\in \Gamma\setminus \Sigma$ and let $\eps$ be a choice of $\P^3$-fibrations for $\sZ^{\ge 2}_{A_0}$.\ Then, there exists a      classical  open neighborhood $\cV$ of $[A_0]$ in $ \LG(\bw3V_6)\setminus \Sigma$ such that
        \begin{itemize}
        \item[\rm(a)]  There is a factorization
 $$
        \rho_\cV \colon \widetilde \cX_\cV\xrightarrow{\ d_\eps\ }  \cX_\cV^\eps \xrightarrow{\ c_\eps\ } \widetilde \cZ^{\ge 2}_\cV,
    $$
    where $\cX_\cV^\eps$  is a complex manifold  and $d_\eps$ induces on each $E_i\subset \widetilde \sX_{A_0}\subset \widetilde \cX_\cV$, the chosen projection $E_i\to \P^3$,
     \item[\rm(b)] The composition 
 $$
        \varphi_\cV \colon \cX_\cV^\eps \xrightarrow{\ c_\eps\ } \widetilde \cZ^{\ge 2}_\cV \lra \cV 
    $$
    is smooth  and
     \begin{itemize}
        \item the fiber of $[A]\in \cV\smallsetminus \Gamma$ is the EPW cube $\widetilde\sZ^{\ge2}_{A}$,
        \item the fiber $\sX_{A}^\eps$ of $[A]\in \cV\cap \Gamma$ is a small resolution of the singular EPW cube $\widetilde\sZ^{\ge2}_{A}$, with exceptional locus a disjoint union of $\P^3$ (one for each singular point of $\widetilde\sZ^{\ge2}_{A}$).
         \end{itemize}
    \end{itemize}
    \end{prop}

Given another choice  $\eps'$ of $\P^3$-fibrations, there exists a commutative diagram
$$
\begin{tikzcd}
    & \widetilde \sX_{A_0}  \ar[ld,"d_\eps"'] \ar[rd,"d_{\eps'}"] & \\ 
    \sX_{A_0}^\eps \ar[rd, "c_\eps"']  \ar[rr, dashed] & &  \sX_{A_0}^{\eps'} \ar[ld,"c_{\eps'}"] \\ 
    &\widetilde \sZ^{\ge 2}_{A_0}, 
\end{tikzcd} $$
where the horizontal dashed arrow is a birational morphism which is the Mukai flop of the   exceptional $\P^3$ for which $\eps_i\ne \eps'_i$.\ Note that we do not claim that $\sX_{A_0}^\eps$ is projective.\ In Section~\ref{sec:ProjResolutionZ2A}, we will show that at least one~$\sX_{A_0}^\eps$ (for a suitable choice of $\eps$)  is indeed projective.

\section{Projective resolutions of $\widetilde\sZ^{\ge2}_A$}\label{sec:ProjResolutionZ2A}

Given a Lagrangian $[A]\in \Gamma\setminus \Sigma$, we described in Section~\ref{sect42} the blowup $\tau\colon \widetilde \sX_A\to \widetilde \sZ^{\ge 2}_A$ of  the singular locus of the EPW cube $\widetilde \sZ^{\ge 2}_A$ and, {\em in the analytic category,} some small resolutions $h_\eps\colon \sX_A^\eps\to \widetilde \sZ^{\ge 2}_A$ together with factorizations
$$
\tau\colon \widetilde\sX_A\xrightarrow{\ d_\eps\ } \sX_A^\eps \xrightarrow{\ c_\eps\ } \widetilde \sZ^{\ge 2}_A. 
$$

The aim of this section is to show that there exists $\eps$ such that the contraction $c_\eps$ is projective.\ We do so by constructing a smooth projective \hk\ variety $\sY_A$ with a small contraction $  \sY_A \to \widetilde \sZ^{\ge 2}_A$  and then showing that the blowup map $\tau$ factors though it.\ Our construction will use properties of the period map for EPW cubes which we now explain.

\subsection{Period maps for EPW cubes and double EPW sextics}

Smooth EPW cubes are \hk\ sixfolds of $\KKK^{[3]}$-type with a polarization of   Beauville--Bogomolov square 4 and divisibility 2.\ There is a
19-dimensional quasi-projective irreducible coarse moduli space ${}^{[3]}\cM_4^{(2)}$ for these polarized sixfolds (see \cite[Theorem~3.5]{deb}), hence a {\em moduli morphism}
$$\mu_3\colon\LG(\bw3V_6)\smallsetminus(\Gamma\cup \Sigma)\lra {}^{[3]}\cM_4^{(2)}
$$
that sends a  Lagrangian $[A]\notin \Gamma\cup \Sigma$ to the  point of the moduli space corresponding to the smooth EPW cube $\widetilde\sZ_A^{\ge2}$ with its canonical polarization.

There is also a 19-dimensional quasi-projective period domain ${}^{[3]}\cP_4^{(2)}$ and a {\em period map}
$$ \wp_3\colon {}^{[3]}\cM_4^{(2)}\lhra {}^{[3]}\cP_4^{(2)} $$
(see Section~\ref{secA1} for more details) which is injective by the Torelli theorem (\cite{huy, mar, ver}).

The situation for double EPW sextics is completely analogous and we describe it only briefly.\ Given a Lagrangian subspace $A\subset \bw3V_6$ not in $\Delta\cup \Sigma$ (see Section~\ref{sec2} for the definition of the divisors $\Delta$  and~$\Sigma$), O'Grady constructed in \cite{og} a smooth double cover of a certain sextic hypersurface in $\P(V_6)$.\ This double cover is a \hk\ fourfold with a canonical polarization of Beauville--Bogomolov square 2 and divisibility 1.\ So, as above, we have a moduli morphism 
$$\mu_2\colon\LG(\bw3V_6)\smallsetminus(\Delta\cup \Sigma)\lra {}^{[2]}\cM_2^{(1)}
$$
 and an injective period map
$$ \wp_2\colon {}^{[2]}\cM_2^{(1)}\lhra {}^{[2]}\cP_2^{(1)}. $$
O'Grady proved in \cite[Proposition~4.8]{og6} that the  composition $  \wp_2 \circ \mu_2$ extends to a regular morphism  
\begin{equation}\label{defp2}
\overline\wp_2\colon \LG(\bw3V_6)\smallsetminus  \Sigma'\lra {}^{[2]}\cP_2^{(1)}, 
\end{equation}
where $\Sigma'\varsubsetneq \Sigma$ is a proper closed subset.

There is a degree-2 finite map
\begin{equation*}\label{defrho}
    \rho\colon {}^{[2]}\cP_2^{(1)}\lra {}^{[3]}\cP_4^{(2)}
\end{equation*}
defined in Equation~\eqref{eqn:def_rho} and the main   theorem of \cite[Corollary~6.1]{kkm} says that 
 \begin{equation}\label{ekkm}
    \rho\circ \wp_2\circ\mu_2=\wp_3\circ \mu_3 
\end{equation}
on $\LG(\bw3V_6)\smallsetminus(\Delta\cup\Gamma\cup \Sigma)$.\ This implies that the  composition $   \wp_3 \circ \mu_3$ extends to a regular morphism  
\begin{equation}\label{defp3}
\overline\wp_3\coloneqq \rho\circ \overline\wp_2\colon \LG(\bw3V_6)\smallsetminus  \Sigma'\lra {}^{[3]}\cP_4^{(2)}. 
\end{equation}

\subsection{Projective resolutions of singular EPW cubes.}\label{sec:ProjRes}
We show here that one of the analytic contractions $c_\eps$ of Proposition~\ref{prop:contraction} is projective.

\begin{prop}\label{prop:ProjSmallRes}
    Let $[A]\in\Gamma\setminus \Sigma$.\ 
    There exists a smooth quasi-polarized projective \hk\ variety $(\sY,H)$ such that  the big and nef line bundle $H$ defines a small contraction 
    $c \colon \sY_A \to \widetilde\sZ^{\ge 2}_A$.
\end{prop}

\begin{proof}
We use the following lemma,   whose proof goes as   the proof of \cite[Theorem~6.1]{dm}.

 \begin{lemm}\label{lemm:dm}
     Given a point $x\in {}^{[3]}\cP_4^{(2)}$, there exists a smooth projective \hk\ sixfold~$Y$ with a big and nef line bundle $H$ of square $4$ and divisibility $2$ with period $x$.
 \end{lemm}

Let $(\sY_A,H)$ be the quasi-polarized \hk\ sixfold associated to the period point $x\coloneqq \overline\wp_3([A])$ given by the lemma.

Consider the local universal deformation $f\colon (\cY,\cH) \to (B,0)$ of the pair $(\sY_A,H)$.\ For~$b$ general in~$B$, the line bundle $\cH_b$ is ample on  $\cY_b$ and  the pair  $(\cY_b,\cH_b)$ defines an element of the moduli space ${}^{[3]}\cM_4^{(2)}$.\ Up to shrinking $B$ to a smooth curve containing $0$, we can suppose that for each $b\in B$, the point $\wp_3([\cY_b])$ is contained in the image of $\overline\wp_3 $ and is, for all $b\neq 0$, the period of a (smooth) EPW cube.

Since the line bundle $\cH$ is $f$-big and $f$-nef, the relative Kawamata--Viehweg  vanishing    theorem (\cite[Theorem 1-2-5 and Remark 1-2-6]{kkm}) implies $\R^{>0}f_*\cH=0$.\ Furthermore, by the relative Kawamata  base-point-free theorem (\cite[Theorem 3-1-1 and Remark 3-1-2]{kkm}), the $\cO_B$-algebra $\bigoplus_{m\ge 0}f_*\cH^m$ is finitely generated and there is a $B$-morphism $\varphi\colon  \cY \to \cZ'$ to its projective spectrum such that the  line bundle $ \cH$ is   the pullback of a relatively ample line bundle on   $\cZ'$.\ The morphism $\phi$   then induces   isomorphisms $\cY_b\isomto \cZ'_b$ for $b\in B\setminus\{0\}$, and   a projective contraction $  \sY_A \to \sZ'$ on   central fibres.\  

Consider the families
\begin{equation}\label{eqn:diag}
   \begin{tikzcd}%\ar[rr, dashed, "\varphi"]
    \widetilde\cZ_{B}^{\ge 2} \ar[dr]& & \cZ'\ar[dl]\\ 
    & B,
\end{tikzcd} 
\end{equation}
 where the morphism $\widetilde\cZ_{B}^{\ge 2}\to B$, with central fiber $\widetilde\sZ^{\ge 2}_A$, is obtained by lifting the curve $B $ to $\LG(\bw3V_6)$ using the period map~$\wp_3$ as in \cite{og6}, and comes with a relatively ample line bundle (which is the pullback of the hyperplane section of $\sZ^{\ge 2}_A\subset \Gr(3, V_6)$ on each fiber).\ 

Over each point of $B\setminus\{0\}$, the fibers of the two morphisms in the diagram~\eqref{eqn:diag} are isomorphic polarized \hk\  manifolds with the same period map hence, by the Torelli theorem, there exists an isomorphism between the two families that respects the polarizations.\ By separatedness of the moduli space of polarized varieties with trivial canonical bundle (see for example \cite[Theorem~2.1]{bouc}), the central fibers $\widetilde\sZ^{\ge 2}_A$ and $\sZ'$ of these two morphisms are isomorphic.

Therefore, we have a projective birational contraction $c \colon \sY_A \to  \widetilde\sZ^{\ge 2}_A$, defined by the sections of a sufficiently large multiple $mH$ of $H$.\ Taking the top self-intersections of both sides of the equality   $c^*H_A=mH$, we obtain  $m=1$, so in fact, $H$ defines $c$.\

  In particular, $\widetilde \sZ^{\ge 2}_A$ is isomorphic outside a closed subset of codimension $\ge 2$ to an open subset of $\sY_A$, which has trivial canonical bundle.\ It follows that its smooth locus $(\widetilde \sZ^{\ge 2}_A)_{\mathrm{sm}}$ also has trivial canonical bundle.

It remains to show that $c$ is small.\ For that, we first   show that the image of the exceptional locus   of $c$ is the singular locus of $\widetilde \sZ^{\ge 2}_A$.\ Consider the restriction  
$$
c' \colon c^{-1}\bigl( (\widetilde \sZ^{\ge 2}_A)_{\mathrm{sm}}\bigr)\lra (\widetilde \sZ^{\ge 2}_A)_{\mathrm{sm}} 
$$
of $c$ to the preimage of the smooth locus of $\widetilde \sZ^{\ge 2}_A$.\ It is a birational morphism between smooth varieties, hence
%hence if $c^{-1](\widetilde \sZ^{\ge 2}_A)_{\mathrm{sm}}$ meets $E$, it is a divisorial contraction. But this implies
$$
K_{c^{-1}((\widetilde \sZ^{\ge 2}_A)_{\mathrm{sm}})} \lin K_{(\widetilde \sZ^{\ge 2}_A)_{\mathrm{sm}}} + \mathrm{Exc}(c').
$$
Since the canonical divisors $K_{c^{-1}((\widetilde \sZ^{\ge 2}_A)_{\mathrm{sm}})}$ and $K_{(\widetilde \sZ^{\ge 2}_A)_{\mathrm{sm}}}$ are trivial, we obtain that $c'$ is an isomorphism.\ Therefore, $c$ is an isomorphism over $(\widetilde \sZ^{\ge 2}_A)_{\mathrm{sm}}$.\ This implies that the image of the exceptional locus of $c$ is the finite set of singular points of $\widetilde \sZ^{\ge 2}_A$, and therefore the contraction is small by \cite[Proposition~(1.4)]{nam}.
%\footnote{\Fr{ I would like to say that it is small since the image of the exceptional locus is the singular locus of $\widetilde\sZ^{\ge 2}_A$ (plus https://arxiv.org/pdf/0912.4981 ), but I'm not sure I know it.}}
 \end{proof} 

 Given $[A]\in \Gamma\setminus\Sigma$, we now prove that the blowup $\tau\colon  \widetilde \sX_A\to \widetilde \sZ^{\ge 2}_A$ of $\widetilde \sZ^{\ge 2}_A$ along its finite singular locus $\sZ^4_A$ factors through the small contraction $c$ constructed above.

\begin{theo}\label{theo:projContraction}
    Let $[A]\in \Gamma\setminus \Sigma$ be a Lagrangian subspace of $\bw3V_6$.\ The resolution $\tau\colon  \widetilde \sX_A\to \widetilde \sZ^{\ge 2}_A$ obtained by blowing up the finite singular locus of the EPW cube $\widetilde\sZ^{\ge 2}_A$ factors through the small contraction $c$ constructed in Proposition~\ref{prop:ProjSmallRes}:
$$
\begin{tikzcd}
    & \widetilde \sX_A  \ar[ld, "d"'] \ar[dd, "\tau"] & \\ 
    \sY_A \ar[rd, "c"']  & &  \\ 
    &\widetilde \sZ^{\ge 2}_A.
\end{tikzcd} 
$$
% $$
% \begin{tikzcd}
%     \widetilde \sX_A \ar[rr, "\tau"] \ar[rd, "d"']& & \widetilde \sZ^{\ge 2}_A\\ 
%     & \sY_A \ar[ru, "c"'].
% \end{tikzcd} 
% $$
The morphism $d\colon\widetilde \sX_A\to \sY_A$ contracts each component $E_i\subset \P^3\times \P^3$ of the exceptional locus of~$\tau$  to a copy of $\P^3$ and the small contraction $c\colon\sY_A\to  \widetilde \sZ^{\ge 2}_A$ contracts each such $\P^3$ to a singular point of $\widetilde \sZ^{\ge 2}_A$.\  

Moreover, there exists another projective variety $\sY'_A$ and a diagram
$$
\begin{tikzcd}
    & \widetilde \sX_A  \ar[ld, "d"'] \ar[rd, "d'"] & \\ 
    \sY_A \ar[rd, "c"']  \ar[rr, dashed] & & \sY'_A \ar[ld, "c'"] \\ 
    &\widetilde \sZ^{\ge 2}_A, 
\end{tikzcd} $$
where $\sY_A\dra \sY'_A$ is the flop of every $\P^3$ in the exceptional locus of $c$.
\end{theo}

We will need the following lemma. 

\begin{lemm}\label{lemma:GlobGen}
   Let $D$ be a divisor on $\widetilde \sX_A$ 
   such that the line bundle $\cO_{E_i}(D\vert_{E_i})$ has nonzero sections for each component $E_i$ of $E$.\ 
    For any   sufficiently ample divisor~$H$ on $\widetilde \sZ^{\ge 2}_A$, the line bundle $\cO_{\widetilde \sX_A}(D+\tau^*H)$ is generated by its global sections. 
\end{lemm}

\begin{proof}
    Consider the divisor $\tau_*D$ on $\widetilde \sZ^{\ge 2}_A$.\ For $H$ sufficiently ample, $\tau_*D+H$ is globally generated on $\widetilde \sZ^{\ge 2}_A$.\ Write $\tau^*(\tau_*D)=D+D'$, where the support of $D'$ is   contained in $E$.\ For $H$ sufficiently ample, we may also suppose that the linear system $|\tau^*H-D'|$ has no base points outside of~$E$.
    
    The first condition on $H$ implies that, for any $x\notin E$, there exists an effective divisor $F$ on $\widetilde \sZ^{\ge 2}_A$ linearly equivalent to $\tau_*D+H$, whose support does not contain $\tau(x)$.\ We then have  $\tau^*F=D+D'+\tau^*H$, hence
    $D +2\tau^*H\lin \tau^*F+\tau^*H -D'$.\ The second condition on $H$ implies that $x$ is not in the base locus of $D +2\tau^*H$.

    It remains to show that $D +2\tau^*H$ has no base-points on $E$.\ The divisor $(D+2\tau^*H )\vert_{E_i}=D\vert_{E_i}$ is effective by hypothesis hence is globally generated on $E_i$ (like every effective divisor on~$E_i$).\ Therefore,  
    it is  enough to show that  we can lift the sections of $(D+2\tau^*H)\vert_E$ to $\widetilde \sX_A$.\ For that,   the vanishing 
     $$H^1(\widetilde \sX_A, D+2\tau^*H -E) = 0 $$
is enough.\ Using the Leray spectral sequence, it is enough to show that
$$H^1\bigl(\widetilde \sZ^{\ge 2}_A, (\tau_*\cO_{\widetilde \sX_A}(D-E))(2\tau^*H )\bigr)=H^0\bigl(\widetilde \sZ^{\ge 2}_A, (R^1\tau_*\cO_{\widetilde \sX_A}(D-E))(2\tau^*H )\bigr)=0.$$
The vanishing of the $H^1$   holds for $H$ sufficiently ample (\cite[Proposition~5.3]{hart}).\ For the vanishing of the $H^0$, we need to show $ R^1\tau_*\cO_{\widetilde \sX_A}(D-E)=0$.\ This follows from
  the theorem on formal functions (\cite[Theorem~11.1]{hart}), noticing that $H^1(E,D  -mE) = 0$ for all $m\ge 0$ (because $E\vert_{E_i}=\cO(-1,-1)$ and, again, $D\vert_{E_i}$ is effective).\end{proof}

\begin{proof}[Proof of the theorem]%[Sasha's proof]
By \cite[Corollary~A.2]{kp23},   small normal partial resolutions of $\widetilde\sZ^{\ge 2}_A$ correspond to nontrivial classes in $\Cl(\widetilde\sZ^{\ge 2}_A)/\Pic(\widetilde\sZ^{\ge 2}_A)$ modulo multiplication by positive integers.\ In particular, the existence of $c$ implies that there exists a Weil divisor class $[D]$ on $\widetilde\sZ^{\ge 2}_A$ which is not Cartier at any singular point    (and such that $\sY_A$ is the blowup of $[D]$).\

%We show that we can extend $\varphi$ to $\widetilde\sX_A$.\ 
As before, we denote by $z_1, \dots, z_r$ the points of~$\sZ^4_A$ and by $E_i$ the component of the exceptional divisor $E$ of $\tau$ over the point $z_i$.
%Since $\widetilde \sZ^{\ge 2}_A$ is separated, if $\varphi$ extend to a morphism defined on $\widetilde X_A$, the extension is unique. 

\noindent{\it Construction of a divisor~$L$ on~$\widetilde \sX_A$ whose restriction   to each $E_i$ is of the form $\cO_{E_i}(a_i,b_i)$ with $a_i\neq b_i$.} The blowup $\tau$ induces an isomorphism $\Cl(\widetilde \sZ^{\ge 2}_A)\simeq \Pic(\widetilde \sX_A)/\bigoplus_i \Z E_i$.\ Let~$L$ be a divisor on~$\widetilde \sX_A$ whose class corresponds to the Weil divisor class $[D]\in \Cl(\widetilde \sZ^{\ge 2}_A)$ and whose support does not contain any $E_i$.\

Choose a neighborhood $\sZ_i\subset \widetilde \sZ^{\ge 2}_A$  of $g^{-1}(z_i)$ that     contains no other singular points  and set $\sX_i\coloneqq \widetilde\sX_A\times_{\widetilde \sZ^{\ge 2}_A} \sZ_i$ and $\sY_i \coloneqq \sY_A\times_{\widetilde \sZ^{\ge 2}_A} \sZ_i$.\ Then,  $\sX_i\to \sZ_i$ is the blowup of the unique singular point $g^{-1}(z_i)$ of $\sZ_i$  and $c_i\colon \sY_i\to \sZ_i$ is a small projective resolution. 

Again by \cite[Corollary~A.2]{kp23}, the contraction $c_i$ corresponds to the nontrivial class of the Weil divisor $D_i\coloneqq D\cap \sZ_i$ in $\Cl(\sZ_i)/\Pic(\sZ_i)$.\ As in \cite[Equation~(1.1)]{kp23}, there is an injection 
$$
j_i\colon \Cl(\sZ_i)/\Pic(\sZ_i)\lhra \Pic(E_i)/\Z (E_i\vert_{E_i}) \simeq \Z,
$$
where the last isomorphism follows from the Lefschetz hyperplane theorem.\ 
%Indeed, the pullback via the inclusion $E_i\subset \sX_i$ induces a map $\Cl(\sZ_i)\simeq \Cl(\sX_i)/\Z E_i \to \Pic(E_i)/\Z (E_i|_{E_i})$, whose kernel are divisors on $\sX_i$ whose support does not meet $E_i$ (that corresponds to Cartier divisors on $\sZ_i$).

Therefore, $\Cl(\sZ_i)/\Pic(\sZ_i)\simeq \Z$ has rank $1$ and is generated by the class of $D_i$.\ Since $j_i$ is injective, the divisor class $j_i([D_i])=[L\vert_{E_i}]$ is not a multiple of $[E_i|_{E_i}] \simeq \cO(-1,-1)$, hence is of the form $\cO_{E_i}(a_i, b_i)$ with $a_i\neq b_i$.

\noindent{\it Construction of smooth projective varieties $\sY_+$ and $\sY_-$ and divisorial contractions $d_\pm\colon \widetilde \sX_A\to \sY_\pm$.} 
 Let $L$ be the divisor constructed above.\ There exist integers $p_i$ and $q_i$ such that $L+p_iE_i$ and $-L+ q_iE_i$ restrict  on $E_i$ to $\cO_{E_i}(c_i, 0)$ and $\cO_{E_i}(0,c_i)$, with $c_i>0$. 

 We set $L_1\coloneqq L+p_1E_1+\cdots+p_rE_r$ and $L_2\coloneqq -L+q_1E_1+\cdots+q_rE_r$.\  These divisors restrict to   positive multiples of $\cO(1, 0)$ and $\cO(0, 1)$ on each $E_i$.\  
By Lemma~\ref{lemma:GlobGen}, after adding to $L_1$ and $L_2$   the pullback of a sufficiently ample divisor on $\widetilde \sZ^{\ge 2}_A$, we can suppose that they are globally generated.\ Therefore, by \cite[Theorem 1 and Corollary 2]{ish}, they define   morphisms $d_\pm\colon {\widetilde \sX_A}\to \sY_\pm$, with $\sY_\pm$ smooth projective, that map each $E_i\subset \P^3\times\P^3$ to either one of the two $\P^3$ and fit into factorizations
$$
\tau\colon \widetilde \sX_A\xrightarrow{\ d_\pm\ } \sY_\pm \xrightarrow{\ c_\pm\ } \widetilde \sZ^{\ge 2}_A,
$$
where $c_+$ and $c_-$ are small projective contractions that are not isomorphic over each $\sZ_i$.

\noindent{\it End of the proof.} Because of the isomorphism $\Cl(\sZ_i)/\Pic(\sZ_i)\simeq \Z$,  the results of \cite[Appendix~A]{kp23} imply that $\sZ_i$ has exactly two projective small resolutions, namely $\sY_i^{+}=\Bl_{[D_i]}(\sZ_i)$ and $\sY_i^{-}=\Bl_{[-D_i]}(\sZ_i)$, and $\sY_i$ is one of those.\ 
Over each $\sZ_i$, the contractions $c_+$ and $c_-$ are two distinct normal small projective resolutions of $\sZ_i$, hence they must coincide with $\sY_i^{+}$ and $\sY_i^{-}$.\ Since $\sY_i$ is one of those, we have a morphism $\sX_i\to \sY_i$.\ In this way, we extend the rational map $d\coloneqq  c^{-1}\circ \tau \colon  \widetilde \sX_A\dra \sY_A$   to each $\sX_i$, hence to $\widetilde \sX_A$,  as a regular morphism.

Finally, $\sY_A$ is the blowup of, say, $[D]$, and \cite[Corollary A.4.]{kp23} shows that   the blowup~$\sY'_A$ of $[-D]$ is the flop of $\sY_A$, and the same argument shows that there exists a morphism $d'\colon \widetilde \sX_A \to \sY'_A$.
\end{proof}

 \section{No associated K3 surface}

As mentioned earlier, O'Grady obtained in \cite{og}  results similar to our Theorem~\ref{theo:projContraction}, but for double EPW sextics: when $[A]\in \Delta\smallsetminus \Sigma$, the   double EPW sextic associated to~$A$ is singular with   small projective resolutions.\ In addition, he constructs a smooth K3 surface $S_A \subset \P^6$ and proves that   the    double EPW sextic is birationally isomorphic to the Hilbert square $S_A^{[2]}$.%\footnote{\rem{Note that O'Grady only proves that, {\em when $S_A$ contains no lines,} one of the  small resolutions is isomorphic to $S_A^{[2]}$.\ When $S_A$ contains   lines, he does not prove that any of the  small resolutions is K\"ahler.}}

In our case, we prove below that, on the contrary,   when $[A]\in \Gamma\smallsetminus \Sigma$ is very general,  the singular  EPW cube $\widetilde \sZ^{\ge2}_A$ is {\em not} birationally isomorphic to any smooth moduli space of  stable sheaves on a K3 surface.\ For this, we use the determination of the image by the period map of the divisor $\Gamma$ obtained in Proposition~\ref{prop:imagediv} and results from \cite{PirOrtiz}.

  %Now that we have identified the period points of the singular EPW cubes $\widetilde\sZ^{\ge2}_A$, for $[A]\in \Gamma\smallsetminus\Sigma$, we can prove that the latter are in general not related to K3 surfaces  in a usual sense.

\begin{prop}\label{prop:NoAssociatedK3}
When $[A]\in \Gamma\smallsetminus\Sigma$ is very general, the singular EPW cube  $\widetilde\sZ^{\ge2}_A$
 is not  birationally isomorphic to any smooth moduli space of stable sheaves on a K3 surface.
\end{prop}

\begin{proof}
   By \cite[Theorem~3.7]{PirOrtiz}, a smooth \hk\ variety $X$ of $\KKK^{[3]}$-type is birationally isomorphic to a smooth moduli space $M_v(S, H)$ for some K3 surface $S$ if and only if $X$ and~$S$ have Hodge isometric transcendental lattices. We will determine these lattices when $X$  is a smooth model of a very general   singular EPW cube  $\widetilde\sZ^{\ge2}_A$ as in the proposition and when $S$ is a (necessarily very general) polarized K3 surface, and prove that they cannot be isomorphic.

  As shown in Proposition~\ref{prop:ProjSmallRes}, when $[A]\in \Gamma\smallsetminus\Sigma$ is very general, the singular EPW cube~$\widetilde\sZ^{\ge2}_A$
 is birationally isomorphic to a smooth quasi-polarized projective \hk\ variety $(\sY_A,H)$ of K3$^{[3]}$-type.

 Let $\Lambda_{\KKK^{[3]}} = \Lambda_{\KKK}\oplus \Z(-4)$ be the lattice associated with \hk\  manifolds of K3$^{[3]}$-type  and let $h$ be a vector of square 4 and divisibility $2$ in $\Lambda_{\KKK^{[3]}}$.\  By Proposition~\ref{prop:imagediv}, the period point of $\sY_A$ is a very general point of the Heegner divisor ${^{[3]}\cD_{4,12}^{(2)}}$, hence the  transcendental lattice of $\sY_A$ is isometric to $K^\perp\subset \Lambda$, where $K\subset \Lambda$ is a discriminant $12$ lattice that contains~$h$.\

    As explained in the proof of Proposition~\ref{prop31}, the lattice $h^\perp$ is of the form $M \oplus \Z k \oplus \Z\ell$, where $M = N\oplus U$ is unimodular and $k$ and $\ell$ have square $-2$, and  the lattice $K^\perp$ is equal to $\beta^\perp\subset h^\perp$ for some vector $\beta\in h^\perp$ of square $12$, divisibility $2$ and class $\beta_*=(1,1)\in A_{h^\perp}$ (up to isometries of $O(h^\perp)$, this vector is unique).\  Therefore, we can take $\beta = 2m + k + l$ with $m = u - v$, where $(u,v)$ is a basis of $U\subset M$.\ The transcendental lattice of $X$ is equal to 
    $$
    T = N \oplus \begin{pmatrix}
            2 & 0 & 1\\
            0 & -4 & -2\\
            1 & -2 & -2
            \end{pmatrix},
    $$
    where a basis of the nonunimodular part is given by $(u+v, k-\ell, v + k)$.
    
    The transcendental lattice of a very general polarized K3 surface of degree $2e$  is isomorphic to $N \oplus U \oplus \Z(-2e)$.\ We show that the lattice $T$ defined above is not of this form  by proving that there exist  no vectors $w\in T$ of divisibility $1$ and square $0$ orthogonal to $N$.\ Such a vector $w$ would be of the form $x(u+v)+y(k-\ell) + z(v+k)$.\ We compute its divisibility and square in terms of $x,y,z$:
    $$
        \div(w) = (2x + z, -4y-2z, x-2y-2z)
    $$
    and 
    $$
        w^2 = 2x^2 -4y^2 -2z^2 + 2xz - 4yz = 2(x^2 -z^2 +xz -2(y^2+yz)).
    $$
    We now impose $w^2 = 0$ and $\div(w) = 1$ and $w^2 = 0$.\ The first condition    implies
    $$
    x^2 -z^2 +xz -2(y^2+yz) = 0.
    $$
    In particular, $x^2 +z^2 + xz  $ is even, hence $x$ and $z$ are both even, which is absurd since $w$ has divisibility 1.
\end{proof}

\begin{rema}
   More generally,  \cite[Corollary~4.4]{bmw} shows  that for $[A]\in \Gamma\smallsetminus\Sigma$ very general, the singular EPW cube  $\widetilde\sZ^{\ge2}_A$
 is not  birationally isomorphic to any moduli space of \emph{twisted} sheaves on a K3 surface. 
\end{rema}

\appendix
\section{Divisors in the period domain ${}^{[3]}\cP_{4}^{(2)}$}\label{appendix}

In this appendix, we define Heegner divisors in period domains and prove some of their properties in cases that concern us. We then   analyze the images in the period domains of the  divisors $\Delta$, $\Gamma$, and $\Sigma$ defined in Section~\ref{sec2} and prove that they are   Heegner divisors. 

 \subsection{Heegner divisors in the period domains ${}^{[2]}\cP_{2}^{(1)}$  and ${}^{[3]}\cP_{4}^{(2)}$} \label{secA1}

We start by introducing some notation, following \cite[Sections~3.9 and 3.10]{deb}. 

Let $\Lambda_{K3^{[m]}}$ be the lattice associated to \hk\ manifolds of $\KKK^{[m]}$-type, and let $\tau$ be a polarization type, namely the $O(\Lambda_{K3^{[m]}})$-orbit of a vector $h\in \Lambda_{K3^{[m]}}$. The period domain of polarized \hk\ manifolds of $\KKK^{[m]}$-type and polarization type $\tau$ is the quotient
$${^{[m]}\cP_\tau} \coloneqq \cD_{h^\perp}/\widehat O(\Lambda_{K3^{[m]}}, h),$$
where $\cD_{h^\perp}\subset \P(h^\perp\otimes \C)$ is the open set of $x\in \P(h^\perp\otimes \C)$ such that $x^2=0$ and $x\cdot \bar x >0$, and $\widehat O(\Lambda_{K3^{[m]}}, h)$ is the group of isometries of $\Lambda_{K3^{[m]}}$ that fix the vector $h$ and act as $\pm \id$ on the discriminant group of $\Lambda_{K3^{[m]}}$. 

When the divisibility  $\div(h)$  of  $h$ is equal to $1$ or $2$, the polarization type $\tau$ of $h$ is uniquely determined by  $h^2$ and $\div(h)$ (see \cite[Remark~3.23]{deb}). Therefore, we  denote by ${}^{[2]}\cP_{2}^{(1)}$ the period space of polarized \hk\ manifolds of $\KKK^{[2]}$-type with polarization of square 2 and divisibility 1, and  by ${}^{[3]}\cP_{4}^{(2)}$ the period space of polarized \hk\ manifolds of $\KKK^{[3]}$-type with polarization of square 4 and divisibility 2.

There exists a finite degree 2 map 
\begin{equation}\label{eqn:def_rho}
    \rho\colon {}^{[2]}\cP_{2}^{(1)}\to {}^{[3]}\cP_{4}^{(2)}.
\end{equation} Indeed,   the period spaces ${}^{[2]}\cP_{2}^{(1)}$ and ${}^{[3]}\cP_{4}^{(2)}$ are obtained as quotients of the same domain $\cD$ by two groups of isometries $\Gamma_2$ and $\Gamma_3$, such that $\Gamma_2$ has index $2$ in $\Gamma_3$ (\cite[Proposition~6.3]{riz}).

\begin{defi}[Heegner divisors] Let $K\subset \Lambda_{K3^{[m]}}$ be a rank-2 lattice of real signature $(1,1)$ that contains $h$.
    The \emph{Heegner divisor} ${^{[m]}\cD_{\tau, K}}$ of ${^{[m]}\cP_\tau}$ is the image   of the hyperplane section $\P(K^\perp\otimes \C)\cap \cD_{h^\perp}$ in the quotient $\cD_{h^\perp}/\widehat O(\Lambda_{K3^{[m]}}, h)$. Moreover, for $d>0$, we set 
    $$
        {^{[m]}\cD_{\tau, d}} \coloneqq \bigcup_{K, \ \disc(K^\perp) = - d} {^{[m]}\cD_{\tau, K}},
    $$
    where $K$ runs over the rank-2 sublattices of $\Lambda_{K3^{[m]}}$ with real signature $(1,1)$ that contain $h$.
    Note that the Heegner divisor ${^{[m]}\cD_{\tau, K}}$   only depends on $\beta^\perp$, where $\beta$ is any nonzero vector in $K \cap h^\perp$. 
\end{defi}

 Following the notation of \cite{dm}, we  denote by $^{[2]}\cD_{2,d}^{(1)}$ the Heegner divisors in the period space ${}^{[2]}\cP_{2}^{(1)}$, and  by $^{[3]}\cD_{4,d}^{(2)}$ the Heegner divisors in the period space ${}^{[3]}\cP_{4}^{(2)}$.
 We compare  these Heegner divisors using the double cover $ \rho\colon {}^{[2]}\cP_2^{(1)}\to {}^{[3]}\cP_4^{(2)}$ defined in~\eqref{eqn:def_rho}.  

Heegner divisors in  ${}^{[2]}\cP_2^{(1)}$ are described in \cite[Proposition~3.29]{deb}: the divisor $^{[2]}\cD_{2,d}^{(1)}$ is nonempty if and only $d=2e$ and $e\not\equiv 3\pmod{4}$. Moreover, if $e$ is even, the divisor $^{[2]}\cD_{2,2e}^{(1)}$ is irreducible, while it has 2 components for $e\equiv 1\pmod{4}$.\ We give a similar description of Heegner divisors in  ${}^{[3]}\cP_4^{(2)}$.

\begin{prop}\label{prop31}
Let $d$ be a positive integer.\ The  Heegner divisor  $^{[3]}\cD_{4, d}^{(2)}\subset {}^{[3]}\cP_4^{(2)}$ is nonempty if and only if $d$ is an even number $2e$ such that $e\not\equiv 3\pmod{4}$.\ When this is the case, $^{[3]}\cD_{4, 2e}^{(2)}$ is irreducible and   is the image  by $\rho$ of the Heegner divisor~$^{[2]}\cD_{2, 2e}^{(1)}\subset {}^{[2]}\cP_2^{(1)}$.
 \end{prop}
 
\begin{proof}  
Recall that ${}^{[3]}\cP_4^{(2)} = \cD_{h^\perp}/\widehat O(\Lambda_{\KKK^{[3]}},h)$, where $h\in \Lambda_{\KKK^{[3]}}$ has square $4$ and divisibility $2$.\ As remarked in \cite[Section~6.1]{riz}, the lattice $h^\perp$ decomposes as 
\begin{equation}\label{eqn:latticeForK3[3]sq4div2}
    h^\perp = M \oplus \Z k \oplus \Z\ell,
\end{equation} 
 where $k$ and $\ell$ are vectors of square $-2$ and $M$ is a unimodular lattice. The discriminant group~$A_{h^\perp}$ is isomorphic to $\Z/2\Z\times \Z/2\Z$.
Moreover the group $\widehat O(\Lambda_{\KKK^{[3]}},h)$ is equal to the whole group of isometries $O(h^\perp)$. Finally, the group $O(h^\perp)$ is generated by its subgroup $\widetilde O(h^\perp)$ of isometries that acts trivially on the discriminant group $A_{h^\perp}$, and  the reflection $r$   with respect to $k+\ell$, which exchanges $k$ and $\ell$ and is the identity on $M$.

Let $K$ be a rank-2 sublattice of $\Lambda_{\KKK^{[3]}}$ of signature (1,1) that contains $h$, and denote by~$\beta$ a primitive generator of $K\cap h^\perp$. We determine when the Heegner divisor defined by $K$ is contained in $^{[3]}\cD_{4, d}^{(2)}$. Since $\beta$ is in $h^\perp$, we can write, with respect to the decomposition \eqref{eqn:latticeForK3[3]sq4div2},
$$
\beta =am+bk + c\ell,
$$
where  $m\in M$ is a primitive vector and $(a,b,c)=1$. Let $2t\coloneqq m^2$.  We   compute $\beta^2 = 2ta^2-2b^2-2c^2$ and $\div(\beta) = (a,2b,2c) = (a,2)$. In particular,~$\beta$ has divisibility $2$ if and only if $2\mid a$, and in this case, either $b$ or $c$ must be odd.% for $\beta$ to be primitive.\ 

Using \cite[Lemma~7.2]{GHS13}, we   compute the discriminant 
\begin{equation}\label{eqn:disc}
 \disc(\beta^\perp) = \frac{-\beta^2\disc(h^\perp)}{\div(\beta)^2}= \frac4{\div(\beta)^2} \cdot 2(-ta^2 + b^2 + c^2),   
\end{equation}
of the lattice   $\beta^\perp\subset h^\perp$,
where we used $\disc(h^\perp)=4$. Since $\div(\beta)\mid 2$, it follows that $\disc(\beta^\perp)$ is equal to $-\beta^2$ when $2\mid a$, and to $-4\beta^2$ when $2\nmid a$. In particular, since $\beta^2$ is even, $\disc(\beta^\perp)$ is an even number $2e$.

Moreover, $e$ cannot be equal to $3$ modulo $4$. Indeed, if   $\disc(\beta^\perp)\equiv 6\pmod{8}$,   then   $a$ must be even and $\disc(\beta^\perp) = -\beta^2 = -2ta^2 +b^2+c^2$, which is never equal to $6$ modulo 8.

Conversely, we now show that for any $e\not\equiv 3\pmod{4}$, there exists a unique $O(h^\perp)$-orbit in $h^\perp$ of vectors $\beta$ such that $\beta^\perp$ has discriminant $2e$. Using \cite[Lemma~3.3]{riz}, this implies that $^{[3]}\cD_{4, 2e}^{(2)}$ is irreducible.

We start by studying the orbit of a primitive vector $\beta\in h^\perp$ of square $\beta^2$ and class $\beta_* = [\beta/\div(\beta)] = (x,y)$ in $A_{h^\perp}=\Z/2\Z\times \Z/2\Z$.\ It is equal to
$$
O(h^\perp)\beta = \widetilde O(h^\perp)\beta \cup \widetilde O(h^\perp)r(\beta).
$$
By Eichler's Lemma, the orbit $\widetilde O(h^\perp)\beta$ is the set of vectors $v$ in $h^\perp$ of square $\beta^2$ and class $v_* \in A_{h^\perp}$ equal to $\beta_*$. Hence, since $r$ acts on $A_{h^\perp}$ by exchanging the components, the orbit $O(h^\perp)\beta$
is the set of primitive vectors $v\in h^\perp$ with square $\beta^2$ and class $v_* \in A_{h^\perp}$  equal to $(x,y)$ or $(y,x)$.\ Therefore, there are at most three $O(h^\perp)$-orbits of primitive vectors of given square, respectively of class $(0,0)$, $(0,1)$ or $(1,0)$, and $(1,1)$, in the discriminant group.

We compute the discriminant of $\beta^\perp$ in each case, using Equation~\eqref{eqn:disc}:
\begin{enumerate}[label = (\alph*)]
\item $\div(\beta)=1$ hence $\beta_* = (0,0)$; in this case, $\disc(\beta^\perp) = -4\beta^2 = 2e$ for some $e\equiv 0\pmod{4}$.
\item  $\div(\beta)=2$ and $\beta_* \in \{ (0,1), (1,0)\}$; in this case, $a$ is even and exactly one among  $b$ and~$c$ is odd.\ Therefore,  $\disc(\beta^\perp) = -\beta^2 = 2e$ for $e=-ta^2 + b^2 + c^2\equiv 1\pmod{4}$.
\item $\div(\beta) = 2 $ and $\beta_* = (1,1)$; in this case, $a$ is even, and $b$ and $c$ are odd, hence $\disc(\beta^\perp) = -\beta^2 = 2e$ for $e=-ta^2 + b^2 + c^2\equiv 2\pmod{4}$.
\end{enumerate}
In particular, $\disc(\beta^\perp)$ determines $\beta^2$ and $\beta_*\in \Z/2\Z\times \Z/2\Z$ up to coordinates swap, hence the $O(h^\perp)$-orbit of $\beta$.

{Finally, to compare Heegner divisors of $^{[2]}\cP_{2}^{(1)}$ and $^{[3]}\cP_{4}^{(2)}$, recall from \cite[Proposition~6.3]{riz} that $^{[2]}\cP_{2}^{(1)} = \cD_{h^\perp}/\widetilde O(h^\perp)$ and the morphism $\rho$ is induced by the action of the reflection $r$ on $^{[2]}\cP_{2}^{(1)}$. Hence,  ${}^{[3]}\cD_{4, 2e}^{(2)}$ is the image of $^{[2]}\cD_{2, 2e}^{(1)}$ by $\rho$.}
\end{proof}

 \begin{rema}\label{rema:HegDivisorAssociatedToDiv}
     The previous proof also shows that the primitive vectors $\beta\in h^\perp$ associated to a Heegner divisor $ {^{[3]}\cD_{4,2e}^{(2)}}$ (where   $e\not\equiv 3\pmod{4}$) satisfy the following: 
 \begin{itemize}
         \item if $e\equiv 0\pmod 4$, the vectors $\beta$ have divisibility $1$ and square $e/2$;
         \item if $e\equiv 1,2 \pmod{4}$, the vectors $\beta$ have divisibility $2$ and square $2e$.
  \end{itemize}
 \end{rema}

 \subsection{Image by $\overline\wp_3$ of the divisors $\Delta$, $\Gamma$, and $\Sigma$}\label{sec:images}
 
In this section, we prove that the images of the divisors $\Delta$, $\Gamma$, and $\Sigma$ defined in Section~\ref{sec31}  by the period map 
\begin{equation*} 
\overline\wp_3\colon \LG(\bw3V_6)\smallsetminus  \Sigma'\lra {}^{[3]}\cP_4^{(2)}. 
\end{equation*}
defined in~\eqref{defp3}   are distinct  Heegner divisors.\ For this, we will use Gushel--Mukai varieties and their periods (for   definitions, the reader may consult~\cite{dk1}). 

\begin{prop}\label{prop:imagediv}
    One has
     \begin{equation*}\label{div3}
    \overline{\overline\wp_3(\Delta\smallsetminus \Sigma') }  = {}^{[3]}\cD_{4,10}^{(2)} \quad,\quad
    \overline{\overline\wp_3(\Gamma\smallsetminus \Sigma') }  = {}^{[3]}\cD_{4,12}^{(2)}
    \quad, \quad
        \overline{ \overline\wp_3(\Sigma\smallsetminus \Sigma')}    =  {}^{[3]}\cD_{4,8}^{(2)} 
. 
\end{equation*}
In particular, the divisors $\Delta$, $\Gamma$, and $\Sigma$ are mutually distinct.
\end{prop}

\begin{proof}
    The   images of these divisors in~${}^{[2]}\cP_{2}^{(1)}$ by the period map~$\overline\wp_2$  were determined in \cite[Remark~5.29]{dk1} and they are Heegner divisors:
\begin{equation*}\label{div}
\begin{array}{rcccl}
   \overline{\overline\wp_2(\Delta\smallsetminus \Sigma') } & =& {}^{[2]}\cD_{2,10}^{\prime\prime(1)} &=&\overline{\{\mbox{periods of GM 4folds containing a $\sigma$-plane}\}}\\
     & & &=&\overline{\{\mbox{periods of GM sixfolds containing a $\P^3$}\}},\\
   \overline{\overline\wp_2(\Gamma\smallsetminus \Sigma') } & =& {}^{[2]}\cD_{2,12}^{(1)} &=&\overline{\{\mbox{periods of GM 4folds containing a $\tau$-plane}\}},\\
    \overline{ \overline\wp_2(\Sigma\smallsetminus \Sigma')}  & =& {}^{[2]}\cD_{2,8}^{(1)} &=&\overline{\{\mbox{periods of nodal GM 4folds}\}} 
\end{array}
\end{equation*}
(where $ {}^{[2]}\cD_{2,10}^{\prime\prime(1)}$ is one of the components of the reducible Heegner divisor ${}^{[2]}\cD_{2,10}^{(1)}$).\ To conclude, it is then enough to apply Proposition~\ref{prop31}, the relation~\eqref{ekkm}, and the fact that the Heegner divisors ${}^{[3]}\cD_{4,10}^{(2)}$, $ {}^{[3]}\cD_{4,12}^{(2)}$, and $ {}^{[3]}\cD_{4,8}^{(2)} $ are mutually distinct.
\end{proof}


\begin{thebibliography}{IKKR}

%  \bibitem[BM]{bama}
%  Bayer, A., Macr\`i, E.,  MMP for moduli of sheaves on K3s via wall-crossing: nef and
% movable cones, Lagrangian fibrations, {\em Invent. Math.} {\bf198} (2014), 505--590.
 
% \bibitem[BD]{bedo}  Beauville, A., Donagi, R., La vari\'et\'e des droites d'une hypersurface cubique de dimension $4$, {\em C. R. Acad. Sci. Paris S\'er. I Math.}  {\bf 301}  (1985),
% 703--706.

 \bibitem[B]{bouc}  Boucksom, S., Remarks on minimal models of degenerations,
unpublished note, 2014, available at \url{http://sebastien.boucksom.perso.math.cnrs.fr/notes/Remarks\_minimal\_models.pdf}
 
\bibitem[BMW]{bmw} Billi, S., Muller, S., Wawak, T.    On birational automorphisms of double EPW-cubes, eprint {\tt arXiv:2405.15510.}
 
\bibitem[D]{deb} Debarre,  O.,    Hyperk\"ahler manifolds, with an appendix with E. Macr\`i, {\em Milan J. Math.} {\bf90} (2022), 305--387.

% \bibitem[DIM]{dims} Debarre, O., Iliev, A., Manivel, L., Special prime Fano fourfolds of degree 10 and index 2, 
%   {\it  Recent Advances in Algebraic Geometry,} 123--155,  C. Hacon, M. Musta\c t\u a, and M. Popa editors, London Math. Soc. Lecture Note Ser. {\bf417}, Cambridge University Press, 2014. 
  
% \bibitem[DK1]{dk} Debarre,  O.,   Kuznetsov,  A.,   Gushel-{M}ukai varieties: classification and birationalities , {\em Algebr. Geom.} {\bf5} (2018), 15--76.

  
\bibitem[DK1]{dk1} Debarre,  O.,   Kuznetsov,  A.,   Gushel–Mukai varieties: linear spaces and
periods, {\em Kyoto J. Math.} {\bf59} (2019), 857--953.
  
  
\bibitem[DK2]{dk2} Debarre,  O.,   Kuznetsov,  A., Double covers of quadratic degeneracy and Lagrangian intersection loci, {\em Math. Ann.}  {\bf378} (2020), 1435--1469.

 
  
\bibitem[DM]{dm}  Debarre, O., Macr\`\i, E., On the period map for polarized hyperk\"{a}hler fourfold,  {\it Int. Math. Res. Not. IMRN} {\bf22} (2019), 6887--6923.
 
  
\bibitem[GH]{gh} Griffiths, P., Harris, J., {\it Principles of algebraic geometry,} Pure and Applied Mathematics, Wiley-Interscience, New York, 1978.

\bibitem[GHS]{GHS13} Gritsenko, V., Hulek, K., Sankaran, G. K. {\it Moduli of {K}3 surfaces and irreducible symplectic manifolds,} Handbook of moduli. {V}ol. {I}, Adv. Lect. Math. (ALM), {\bf24} (2014), 459--526.


\bibitem[H]{hart} Hartshorne, R.,  {\it Algebraic Geometry,}
Graduate Texts in Mathematics, Springer-Verlag, New York-Heidelberg, 1977.

 \bibitem[H]{huy} Huybrechts, D.,  A Global Torelli theorem for hyperk\"ahler manifolds  (after Verbitsky), S\'eminaire Bourbaki 2010/2011, Exp.\ n$^{\rm o}$~1040, {\it Ast\'erisque} {\bf348} (2012), 375--403.

%\bibitem[HT]{hats} Hassett, B., Tschinkel, {\em Asian J. Math.} {\bf14} (2010),  303--322.
 
 \bibitem[IKKR]{ikkr}
Iliev, A., Kapustka, G., Kapustka, M., Ranestad, K.,
EPW cubes, {\em J. reine Angew. Math.}  {\bf748} (2019), 241--268.

 \bibitem[KKM]{kkm}
 Kapustka, G., Kapustka, M., Mongardi, G.,
EPW sextics vs EPW cubes, eprint {\tt arXiv:2202.00301.}

\bibitem[KP]{kp23}
Kuznetsov, A., Prokhorov, Y.,
One-nodal Fano threefolds with Picard number one, eprint {\tt arXiv:2312.13782.}


\bibitem[I]{ish} Ishii, S., 
Some projective contraction theorems, {\em Manuscripta Math.} {\bf22} (1977), 343--358.
 

 \bibitem[M]{mar}  Markman, E.,  A survey of Torelli and monodromy results for holomorphic-symplectic varieties, in {\it Complex and differential geometry,} 257--322,
Springer Proc. Math. {\bf8}, Springer, Heidelberg, 2011. 

 \bibitem[N]{nam}  Namikawa, Y.,  
Deformation theory of singular symplectic $n$-folds,  {\it  Math. Ann.}  {\bf319} (2001), 597--623.
 
 
\bibitem[O'G1]{og2}	O'Grady, K.,  EPW-sextics: taxonomy, {\it Manuscripta Math.}  {\bf138} (2012), 221--272.


\bibitem[O'G2]{og}  O'Grady, K., Double covers of EPW-sextics, {\em Michigan Math. J.}  {\bf62} (2013), 143--184.

\bibitem[O'G3]{og6}	O'Grady, K.,    Periods of double EPW-sextics,  {\it Math. Z.} {\bf 280} (2015),  485--524.

\bibitem[PO]{PirOrtiz}  Piroddi, B., Ortiz, A.,  On the transcendental lattices of Hyperk\"ahler manifolds, eprint {\tt arXiv:2308.12869.}

 \bibitem[R]{riz}
Rizzo, F.,
Groups acting on moduli spaces of hyper-K\"ahler manifolds, {\it Milan J. Math} {\bf 92} (2024), 169--193.
%eprint {\tt arXiv:2304.05480.}

%\bibitem[R2]{rizz} Rizzo, F., $\tau$-planes.

 %\bibitem[S]{song} 
 %Song, J., Geometry of hyperk\"ahler manifolds, Ph.D. thesis, 2022.

  \bibitem[V]{ver} Verbitsky, M.,   Mapping class group and a global Torelli theorem for hyperk\"ahler manifolds.
Appendix A by Eyal Markman,
{\it  Duke Math. J.}  {\bf 162} (2013), 2929--2986. 

\end{thebibliography}
\end{document}